\documentclass[reqno,a4paper]{amsart}
\usepackage[utf8]{inputenc}

\usepackage{amsmath,amssymb}
\usepackage{mathrsfs}
\usepackage{yhmath}

\usepackage{xcolor}

\usepackage{booktabs}

\newtheorem{theorem}{Theorem}
\newtheorem{proposition}{Proposition}
\newtheorem{corollary}{Corollary}
\newtheorem{lemma}{Lemma}

\theoremstyle{remark}
\newtheorem{remark}{Remark}

\newcommand{\de}[0]{\mathrel{\mathop:}=}
\newcommand{\ed}[0]{=\mathrel{\mathop:}}
\newcommand{\ie}[0]{\mathrm{i}}
\newcommand{\dif}[1]{\mathrm{d}#1}
\newcommand{\R}[0]{\mathbb{R}}

\newcommand{\N}[0]{\mathbb{N}}

\DeclareMathOperator{\Cin}{\mathrm{Cin}}

\title[Explicit and unconditional results on gaps between zeroes of $\zeta$]{Some explicit and unconditional results on gaps between zeroes of the Riemann zeta-function}
\author[A.~Simoni\v{c}]{Aleksander Simoni\v{c}}
\address{School of Science, The University of New South Wales (Canberra), ACT, Australia}
\email{a.simonic@student.adfa.edu.au}
\author[T.~Trudgian]{Timothy S.~Trudgian}
\address{School of Science, The University of New South Wales (Canberra), ACT, Australia}
\email{t.trudgian@adfa.edu.au}
\author[C.~L.~Turnage-Butterbaugh]{Caroline L.~Turnage-Butterbaugh}
\address{Carleton College, Northfield, MN, USA}
\email{cturnageb@carleton.edu}
\date{\today}
\thanks{TST is supported by ARC    DP160100932 and FT160100094; CLTB is partially supported by NSF DMS-1901293 and NSF DMS-1854398 FRG}

\definecolor{pink}{rgb}{1,.2,.6}
\definecolor{orange}{rgb}{0.7,0.3,0}
\definecolor{blue}{rgb}{.2,.6,.75}
\definecolor{green}{rgb}{.4,.7,.4}
\definecolor{purple}{RGB}{127,0,255}

\begin{document}

\begin{abstract}
We make explicit an argument of Heath-Brown concerning large and small gaps between nontrivial zeroes of the Riemann zeta-function, $\zeta(s)$. In particular, we provide the first unconditional results on gaps (large and small) which hold for a positive proportion of zeroes. To do this we prove explicit bounds on the second and fourth power moments of $S(t+h)-S(t)$, where $S(t)$ denotes the argument of $\zeta(s)$ on the critical line and $h \ll 1 / \log T$. We also use these moments to prove explicit results on the density of the nontrivial zeroes of $\zeta(s)$ of a given multiplicity.
\end{abstract}

\maketitle

\section{Introduction}

Let $\zeta(s) = \sum_{n=1}^{\infty}n^{-s}$ be the Riemann zeta-function, and write its nontrivial zeroes as $\rho = \beta + \ie\gamma$, where $\beta, \gamma \in \mathbb{R}$. Let $0 < \gamma_1 \le \gamma_2 \le \cdots \le \gamma_n \le \cdots$ denote the ordinates of the nontrivial zeros of $\zeta(s)$ in the upper half-plane. Since
\[
N(T) \de \sum_{0 < \gamma \le T}1 \sim \frac{T}{2\pi}\log T
\]
as $T \to \infty$, the gap between consecutive zeroes $\gamma_{n+1}-\gamma_n$ is $2\pi/\log{\gamma_n}$ on average. Following \cite{CGG85}, define
\begin{equation*}
 D^+(\alpha) \de \limsup_{T\to \infty}D(\alpha, T)\quad \text{and} \quad D^-(\alpha) \de \liminf_{T\to \infty}D(\alpha, T),
\end{equation*}
where
\[
D(\alpha, T) \de \frac{1}{N(T)}\sum_{\substack{0 < \gamma \le T \\ \gamma_{n+1}-\gamma_n\le 2\pi\alpha/\log T}}1.
\]
Note that if $D^-(\mu)>0$ for some $\mu<1$, then a positive proportion of the gaps between consecutive zeroes have length less than $\mu$ times the average spacing. On the other hand, if $D^+(\lambda)<1$ for some $\lambda>1$, then a positive proportion of the gaps between consecutive zeroes have length greater than $\lambda$ times the average spacing. Selberg \cite{SelbergContrib} was the first to obtain unconditionally that such $\mu, \lambda$ exist, however he never published his proof. Fujii \cite{FujiiDiff} also made this observation, and Heath-Brown gives a proof in \cite[Section 9.26]{Titchmarsh}. See the introduction of \cite{CT18} for a discussion concerning this history. Since Selberg's and Fujii's observations, there have been numerous explicit values computed regarding small and large gaps between zeroes of $\zeta(s)$. Let us mention only the most recent results. Define
\[
\mu_0 \de \liminf_{n\to \infty}\frac{\gamma_{n+1}-\gamma_n}{2\pi/\log \gamma_n} \quad \text{and} \quad \lambda_0 \de \limsup_{n\to \infty}\frac{\gamma_{n+1}-\gamma_n}{2\pi/\log \gamma_n},
\]
and note that, trivially, $\mu_0 \le 1 \le \lambda_0$. It is conjectured that $\mu_0 = 0$ (via Montgomery's pair correlation conjecture) and $\lambda_0 = \infty$ (via random matrix theory). Selberg's and Fujii's observations imply $\mu_0 < 1 < \lambda_0$. On the Riemann Hypothesis (RH), Preobrazhenski\u{i} \cite{Pre16} proved that $\mu_0\le0.515396$; Bui--Milinovich \cite{BM18} proved that on RH  $\lambda_0\ge3.18$.
These are the current best results that hold for infinitely many pairs of zeroes under RH. Regarding results that hold for a positive proportion of zeroes, Wu \cite{Wu14} improved previous results by Conrey et al.~\cite{CGG85} and Soundararajan \cite{Sou96}, that on RH one can take $\lambda>1.6989$ and $\mu<0.6553$.  Results holding for a positive proportion of zeroes under additional assumptions can be found in recent work of Chirre et al.\ \cite{CHdL20}.

The main goal of this paper is to provide the first explicit unconditional results for $\mu_0$ and $\lambda_0$.
\begin{theorem}
\label{thm:Main} If $\lambda < 1 + c_0$, then $D^+(\lambda) < 1- c_1$. Thus $\lambda_0 > 1 + c_0$. If $ \mu > 1-(2c_0c_1)/(1-2c_1)$, then $D^-(\mu) > c_2$. Thus $\mu_0 <1-(2c_0c_1)/(1-2c_1)$.
Here
\begin{equation*}
c_0 \de \frac{\pi\exp(4.3)}{\exp(\exp(30.76))}
\end{equation*}
\begin{equation*}
c_1 \de \frac{1}{16e^{99.8}}\left( 1 + \frac{\pi\exp(4.3)}{\exp(\exp(30.76))} - \lambda\right)^2
\end{equation*}
and
\begin{equation*}
c_2\de \frac{(1-2c_1)\mu+2\lambda c_1 -1}{2\mu}.
\end{equation*}
\end{theorem}

To prove Theorem \ref{thm:Main}, we make explicit the argument of Heath-Brown \cite[Section 9.26]{Titchmarsh} which allows one to relate consecutive gaps $\gamma'-\gamma$ to intervals of $[T,2T]$ free of zeroes. To our knowledge this is the first use of this method to obtain explicit results on gaps. The key inserts are explicit bounds on the first, second, and fourth power moments of $S(t+h)-S(t)$, where $S(t)$ denotes the argument of $\zeta(s)$ on the critical line and $h \ll 1 / \log T$.  Here, for $t\in\R$ distinct from the ordinate of any zero of $\zeta(s)$,
\begin{equation*}
\label{eq:Sfunction}
S(t)=\frac{1}{\pi}\arg{\zeta\left(\frac{1}{2}+\ie t\right)},
\end{equation*}
and
\[
S(t)=\lim_{\varepsilon\downarrow0}\frac{1}{2}\big(S(t+\varepsilon)+S(t-\varepsilon)\big),
\]
otherwise, where $\arg{\zeta(1/2+\ie t)}$ is defined by the continuous extension of $\arg{\zeta(s)}$ along straight lines connecting $s=2$, $s=2+\ie t$, $s=1/2+\ie t$, and oriented in this direction, where $\arg{\zeta(2)}=0$. Furthermore, $S(t)$ is a piecewise smooth function with only discontinuities at the imaginary parts of $\rho$, on every interval of continuity is decreasing, and at a point of discontinuity $S(t)$ makes a jump equal to the multiplicity of the zeros on the line $\sigma+it, 1/2\le \sigma \le 2$. It is also known that $S(t)=O\left(\log{t}\right)$ and
\begin{equation}
\label{eq:S1}
S_1(T) = \int_{0}^{T}S(t)\dif{t} = O\left(\log{T}\right).
\end{equation}
Define the $n$-power moment of $S(t+h)-S(t)$ by
\[
\mathcal{J}_{n}(H,h) \de \int_{T}^{T+H} \left|S\left(t+h\right)-S(t)\right|^n \dif{t}.
\]
Tsang has given the following estimate for $\mathcal{J}_{2k}(H,h)$.

\begin{theorem}[Tsang, {\cite[Theorem 4]{Tsang}}]
\label{thm:tsang}
Let $a>1/2$, $T^{a}<H\leq T$ and $0<h<1$. For any positive integer $k$ we have
\begin{flalign*}
\mathcal{J}_{2k}(H,h) = \frac{(2k)!}{2^k\pi^{2k}k!}H\log^{k}&{\left(2+h\log{T}\right)} \\
&+ O\left(H(ck)^k\left(k^k+\log^{k-\frac{1}{2}}{\left(2+h\log{T}\right)}\right)\right),
\end{flalign*}
where $c$ is a positive constant.
\end{theorem}

This slightly improves previous bounds obtained by Fujii, see \cite{FujiiZerosI} and \cite{Fujii}. To prove Theorem \ref{thm:Main}, we require an explicit version of Tsang's result for $\mathcal{J}_2(T,h)$ and $\mathcal{J}_4(T,h)$. In particular, we prove that

\begin{theorem}
\label{umpires}
Let $M\de \frac{1}{2\pi}e^{e^{30.76}}$ and $1\leq\lambda\leq2$. Then there exists $T_0>0$ such that
\[
\mathcal{J}_1\left(T,\frac{2\pi M\lambda}{\log{T}}\right) \geq M_1 T, \quad \mathcal{J}_2\left(T,\frac{2\pi M\lambda}{\log{T}}\right) \geq M_2 T, \quad
\mathcal{J}_4\left(T,\frac{2\pi M\lambda}{\log{T}}\right) \leq M_3 T,
\]
for all $T\geq T_0$, where $M_1\de e^{4.3}$, $M_2\de e^{22.49}$ and $M_3\de e^{58.87}$.
\end{theorem}

The moments of $S(t+h)-S(t)$ are also intimately connected with the density of the nontrivial zeroes of given multiplicity. We show the following.

\begin{theorem}
\label{thm:multMain}
Let $N_j(T)$ be the number of $\rho$ with $0<\gamma\leq T$ and multiplicity $j\geq1$. Then there exists $T_0>0$ such that
\[
\frac{N_j(T)}{N(T)} \leq 1.014\cdot e^{-6.459\cdot10^{-7}j}
\]
for every $j\in\N$ and $T\geq T_0$.
\end{theorem}

It is conjectured that all of the zeroes of the zeta-function are simple. This has been confirmed for at least 40.75\% of zeroes, that is $\liminf_{T\to\infty}N_{1}(T)/N(T)\geq 0.4075$, see \cite[p.\ 11]{PRZZ}. From this it follows that
\begin{equation}
\label{chair}
\frac{N_{j}(T)}{N(T)}\leq \frac{1-0.4075}{j},
\end{equation}
for all $j>1$. While some improvements on \eqref{chair} are possible by considering unconditional bounds on distinct zeroes of $\zeta(s)$, see \cite{Farmer} for example, we remark that Theorem \ref{thm:multMain} gives the first explicit improvement upon \eqref{chair} when $j\geq 2.8\cdot 10^{7}$.

The outline of this paper is as follows. In Section \ref{method} we use the method of Selberg and Fujii as described by Heath-Brown \cite{Titchmarsh} to prove Theorem \ref{thm:Main} under the assumption of Theorem \ref{umpires}. In Section \ref{pitch} we give an explicit approximation of $S(t)$  by means of a trigonometric polynomial. This is used in Section \ref{stumps} to give explicit bounds for moments of $S(t+h) - S(t)$ which leads to the proof of Theorem \ref{umpires}. Finally, in Section \ref{densityproof} we combine results from Sections \ref{pitch} and \ref{stumps}, to give the proof of Theorem \ref{thm:multMain}.

\section{Proof of Theorem \ref{thm:Main}}
\label{method}

Assuming the validity of Theorem \ref{umpires} we may prove Theorem \ref{thm:Main} using the method of Selberg and Fujii as described by Heath-Brown in \cite[Section 9.26]{Titchmarsh}. Define
\begin{equation*}
\mathcal{D}^+_\lambda(T)\de \left\{n\in\N\colon T \le \gamma_n \le 2T, \gamma_{n+1}-\gamma_n\ge \frac{2\pi\lambda}{\log T}\right\}
\end{equation*}
and 
\begin{equation*}
\mathcal{D}^-_\mu(T)\de \left\{n\in\N\colon T \le \gamma_n \le 2T, \gamma_{n+1}-\gamma_n\le \frac{2\pi\mu}{\log{(2T)}}\right\}.
\end{equation*}
For technical reasons which will become clear in the proof,  we divide $2\pi\mu$ by $\log{2T}$ in the definition of $\mathcal{D}^-_\mu(T)$ so that
\[
D\left(\mu,2T\right) := \frac{\left|\left\{n\in\N\colon 0<\gamma_n\leq 2T,\gamma_{n+1}-\gamma_n\leq 2\pi\mu/\log{(2T)}\right\}\right|}{N(2T)} \geq \frac{\left|\mathcal{D}_{\mu}^{-}(T)\right|}{N(2T)}.
\]
Heath-Brown shows that there exists $\lambda >1$ such that $\mathcal{D}^+_\lambda(T) \gg  N(T)$, from which the large gaps result follows. Using the existence of such $\lambda$, he then obtains the small gaps result by proving that there exists $0<\mu <1$ such that $\mathcal{D}^-_\mu \ll  N(T)$. To prove Theorem $\ref{thm:Main}$, we show the existence of $\lambda >1$, $\mu < 1$, $c_1>0,$ and $c_2>0$ such that $\left|\mathcal{D}_{\lambda}^{+}(T)\right|/N(2T)\geq c_1$ and $\left|\mathcal{D}_{\mu}^{-}(T)\right|/N(2T)\geq c_2$. 

\begin{remark}
While we wish to explore an explicit result on gaps between zeroes, we have the luxury of performing all of these calculations for $T$ sufficiently large. For example, if we write $T + O(T/\log T) \geq T(1-\varepsilon)$ for sufficiently small $\varepsilon$ we are spared the ordeal of computing the implied constant in the $O(\cdot)$ notation. To ease exposition, these $\varepsilon$'s may not necessarily be the same in each step in what follows, and we shall not repeat the condition that $T$ be sufficiently large.
\end{remark}

\subsection{Proof of large gaps result.}

Let $\lambda \ge 1$, and let $I$ be a subset of $[T,2T]$ on which $N(t+2\pi\lambda/\log T) = N(t)$ and thus free of zeroes of $\zeta(s)$. In $[T,2T]$, the average gap between consecutive zeroes is $2\pi/\log T$, so if such a subset $I$
exists with $\lambda >1$, it follows that there exists a pair of consecutive zeroes $\gamma,\gamma'$ whose difference is at least $2\pi\lambda/\log T$, i.e., larger than the average spacing. The bridge connecting $m(I)$, the measure of $I$, and $|\mathcal{D}^+_\lambda(T)|$ begins with the observation that
\begin{equation}
\label{bridge}
m(I) \le \sum_{n\in \mathcal{D}^+_\lambda(T)} (\gamma_{n+1}-\gamma_n) +O(1).
\end{equation}
This is true because we have
\[
I \subseteq \left[T,\gamma\right) \cup \bigcup_{n\in \mathcal{D}^+_\lambda(T)} \left[\gamma_n,\gamma_{n+1}\right)
\]
for the first $\gamma$ in $[T,2T]$, and gaps between zeroes are bounded,\footnote{Indeed, see \cite{Sim18} for a proof that the gap between the first two zeroes $\gamma_{1}=14.13\ldots$ and $\gamma_{2}=21.02\ldots$ is the largest of all gaps between consecutive zeroes.} therefore $\gamma-T=O(1)$. We now build off of \eqref{bridge} in two directions. The first and more delicate step is to obtain an explicit lower bound on $m(I)$ of size $T$. Note that on $I$ we have $N(t+2\pi\lambda/\log T) -N(t)=0$. Therefore, to understand $m(I)$ we will study how the difference  $N(t+2\pi\lambda/\log T) -N(t)$ varies for $t \in[T,2T]$. We can reframe this in terms of $S(t)$, the argument of $\zeta(s)$ on the critical line, via the Riemann-von Mangoldt formula
\begin{equation}\label{gamba}
N(t) = \frac{1}{2\pi}t\log t - \frac{1+\log 2\pi}{2\pi}t+ \frac{7}{8} + S(t) + R(t),
\end{equation}
where $R(t) = O(1/t)$ as $t\to \infty$. If $t\in [T, 2T]$ and $0\leq h\ll 1/\log{T}$, then \eqref{gamba} gives
\begin{equation}\label{shifts}
N(t+h) - N(t) - \frac{h\log T}{2\pi} = S(t+h) - S(t) + O\left(\frac{1}{\log T}\right).
\end{equation}
With the left-hand side of \eqref{shifts} in mind, set
\begin{equation}\label{delta}
\delta(t,\lambda)\de N\left(t+\frac{2\pi\lambda}{\log T}\right)-N(t)-\lambda,
\end{equation}
and note that for any $\lambda$, the function $\delta(t,\lambda)$ gives the discrepancy between the actual number of zeroes in $\left(t, t + 2\pi\lambda/\log T\right)$ and the expected number. The estimation of averages of $\delta(t, \lambda)$ will be critical in establishing our main result. Note that\footnote{For $t\notin I$,  note that $N(T+ h) - N(T) \ge 1$, and then consider separately the cases  $N(T+ h) - N(T) \ge \lambda$ and $N(T+ h) - N(T) \le \lambda$.}
\begin{equation*}
|\delta(t,\lambda)| = \begin{cases}
\delta(t,\lambda)+ 2\lambda & \text{if } t \in I,\\
\delta(t,\lambda)+ 2\lambda - 2\lambda, & \text{if } t \notin I,
\end{cases}
\end{equation*}
which implies
\begin{equation}\label{deltaint}
\int_{T}^{2T}|\delta(t,\lambda)|\dif{t} = \int_{T}^{2T}\delta(t,\lambda)\dif{t}+(2\lambda-2)T+2m(I).
\end{equation}
Let $h=2\pi \lambda M/\log T$ with $M$ as in Theorem \ref{umpires}, and observe that
\begin{equation*}
N(t+h)-N(t)-\frac{h\log T}{2\pi} = \sum_{m=0}^{M-1}\delta\left(t+\frac{2\pi m \lambda}{\log T},\lambda \right).
\end{equation*}
Integrating both sides and performing a change of variable, we have
\begin{align*}
\int_{T}^{2T}\left|N(t+h)-N(t)-\frac{h\log T}{2\pi}\right|\dif{t} &= \int_{T}^{2T}\left|\sum_{m=0}^{M-1}\delta\left(t+\frac{2\pi m \lambda}{\log T},\lambda \right)\right|\dif{t} \\
&=\sum_{m=0}^{M-1} \int_{T+2\pi m \lambda / \log T}^{2T+2\pi m \lambda / \log T} |\delta(t,\lambda)|\dif{t}\\
& = M\int_{T}^{2T}|\delta(t,\lambda)|\dif{t} + O(1).
\end{align*}
On the other hand, by Theorem \ref{umpires} we have
\begin{equation*}
\int_{T}^{2T} \left| N(t+h) - N(t) - \frac{h\log T}{2\pi}\right|\dif{t} \geq \left(M_{1}-\varepsilon\right)T,
\end{equation*}
where $M_1$ is given in the statement of Theorem \ref{umpires}. Thus
\begin{equation*}
\label{lute}
\left(M_{1}- \varepsilon\right)T \leq  M \int_{T}^{2T} \left|\delta\left(t, \lambda\right)\right|\dif{t} + O(1).
\end{equation*}
Rearranging, and taking into account the $O(1)$, we find
\begin{equation}
\label{harp}
\int_{T}^{2T}|\delta(t, \lambda)|\dif{t} \geq \left(\frac{M_{1}}{M} - \varepsilon\right)T.
\end{equation}
By \eqref{deltaint} and \eqref{harp}, we have
\begin{equation}\label{oud}
\left(\frac{M_{1}}{M} - \varepsilon\right)T \leq \int_{T}^{2T} \delta(t, \lambda)\dif{t} + (2\lambda -2)T + 2m(I).
\end{equation}
Recalling the definition \eqref{delta} of $\delta(t,\lambda)$ and applying \eqref{shifts} and \eqref{eq:S1}, we have that the first term on the right-hand side of \eqref{oud} is $\ll T/\log T$. It follows that
\begin{equation*}
m(I) \geq \left(\frac{M_{1}}{2M} - (\lambda -1) - \varepsilon\right)T.
\end{equation*}
Therefore, by \eqref{bridge}, we have
\begin{equation}\label{bridge2}
 \left(\frac{M_{1}}{2M} - (\lambda -1) - \varepsilon\right)T \le \sum_{n\in \mathcal{D}^+_\lambda(T)} (\gamma_{n+1}-\gamma_n) +O(1).
\end{equation}
Squaring both sides of \eqref{bridge2} and applying Cauchy--Schwarz, it follows that
\begin{equation}\label{biggerbridge}
T^2\left(\frac{M_{1}}{2M} - (\lambda -1) - \varepsilon\right)^2  \le 2|\mathcal{D}^+_\lambda(T)|\sum_{n\in\mathcal{D}^+_\lambda(T)} (\gamma_{n+1}-\gamma_n)^2 +O(1).
\end{equation}
Korol\"{e}v \cite{Korolev} has shown that
\begin{equation}\label{kor}
\sum_{T/2\le \gamma_{n}\leq T} \left(\gamma_{n+1} - \gamma_{n}\right)^{2} \leq K\frac{N(T)}{\log^{2}{T}},
\end{equation}
where $K=8\pi^2e^{99.8}$. 
Combining \eqref{biggerbridge} with \eqref{kor} we finally have
\begin{equation}
\label{horn}
\frac{\left|\mathcal{D}_{\lambda}^{+}(T)\right|}{N(2T)} \geq \frac{\pi^{2}}{2K} \left(\frac{M_{1}}{2M} + 1 - \lambda\right)^{2} - \varepsilon \ed c_1.
\end{equation}
Hence, to ensure the positivity of $c_1$, we can take $\lambda<1+M_1/(2M)$, e.g., we may take $\lambda = 1 + 397\cdot 10^{-9.93\cdot10^{12}}$. Note that the size of $\lambda$ does not depend on the constant $K$ in Korol\"ev's result, however the size of the proportion of such large-gaps does.

\subsection{Proof of small gaps result}

A unique feature of this method (compared, for example, to the method of \cite{MO} and \cite{CGG84}) is that the result on small gaps is a consequence of the result on large gaps. Similar to \eqref{bridge}, the starting point is to consider the sum
\[
\sum_{T \le \gamma_n \le 2T} (\gamma_{n+1}-\gamma_n).
\]
In this case the telescoping sum is taken over all zeroes in $[T,2T]$, and thus equals $T + O(1)$.  On the other hand, take $0 < \mu < 1$ constant and note that
\begin{flalign*}
\sum_{T \le \gamma_n \le 2T} \left(\gamma_{n+1}-\gamma_n\right) &\ge \sum_{\substack{\gamma_{n+1}-\gamma_n > \frac{2\pi\mu}{\log{(2T)}} \\ T \le \gamma_n \le 2T}} \left(\gamma_{n+1}-\gamma_n\right) \\
&\ge \left(\sum_{\substack{\frac{2\pi\mu}{\log{(2T)}}<\gamma_{n+1}-\gamma_n < \frac{2\pi\lambda}{\log{(2T)}} \\ T \le \gamma_n \le 2T}} + \sum_{\substack{\gamma_{n+1}-\gamma_n \ge \frac{2\pi\lambda}{\log{(2T)}} \\ T \le \gamma_n \le 2T}}\right) \left(\gamma_{n+1}-\gamma_n\right) \\
&\geq \frac{2\pi\mu\left(N(2T)-N(T)-\left|\mathcal{D}^-_\mu(T)\right|-\left|\mathcal{D}^+_\lambda(T)\right|\right)+2\pi\lambda\left|\mathcal{D}^+_\lambda(T)\right|}{\log (2T)}. 
\end{flalign*}
Since $N(2T)-N(T)=\frac{1}{2\pi}T\log{(2T)}+O(T)$, we find
\begin{align*}
T &\ge \mu T - \frac{2\pi\mu}{\log{(2T)}}\left|\mathcal{D}^-_\mu(T)\right| + \frac{2\pi(\lambda-\mu)}{\log{(2T)}}\left|\mathcal{D}^+_\lambda(T)\right| + O\left( \frac{T}{\log{(2T)}}\right).
\end{align*}
Using $c_1$, defined in \eqref{horn}, recalling the convention $T + O(T/\log{(2T)}) = T(1-\varepsilon)$, and rearranging, we find
\begin{equation*}
\frac{|D^-_\mu(T)|}{N(2T)} \ge \frac{\lambda - \mu}{\mu}\cdot c_1 - \frac{1-\mu+\varepsilon}{2\mu(1-\varepsilon)} \geq \frac{\left(1-2c_1\right)\mu + 2\lambda c_1 - 1}{2\mu} - \varepsilon \ed c_2.
\end{equation*}
Thus, to assure positivity of $c_2$, we can take
\[
\mu > 1 - \frac{M_1}{M}\frac{c_1}{1-2c_1}.
\]

\subsection{Limitations of the method}

It is natural to ask whether we could improve substantially on the bounds given in Theorem \ref{thm:Main}.
We note that the critical quantity for $\lambda$ is $M_{1}/M$. That $M$ is forced to be so large is a direct consequence of the size of a constant appearing in the proof of Theorem \ref{umpires}, namely $C_{2}$ in (\ref{lager}). A thorough reworking  of some of the lemmas in \cite{KaratsubaKorolevArgument} and \cite{KaratKor} may yield some further improvements.

The exact limit of such improvements is not clear. A best-possible scenario (almost certainly too good to be true) is one in which there are no error terms in Theorem \ref{thm:difference}: that is, where $C_{2} = C_{3} = 0$. It therefore follows that we may take $M=1$ and
\begin{equation}\label{fish}
M_{1} = \sqrt{\frac{\log(1 + 2\pi \lambda)}{3\pi^{2}}}.
\end{equation}
We can therefore solve for the condition $\lambda < 1 + M_{1}/2$, where $M_{1}$ is in (\ref{fish}). This shows that one may take any $\lambda < 1.1286\ldots$. The actual limit of our method is, in all likelihood, much smaller than this.

We conclude by noting that any improvement on the value of $L$ in the zero-density estimate in (\ref{eq:szde}) has only a minimal effect on the value of $\lambda$. For example, on RH we have $L=0$, which is not enough to improve the bound on $\lambda$ to $1 + 10^{-10^{12}}$.

\section{Explicit approximation of $S(t)$ by segments of a Dirichlet series}
\label{pitch}

Let $N(\sigma,T)$ be the number of the nontrivial zeroes $\rho = \beta + \ie\gamma$ of $\zeta(s)$ with $\beta>\sigma\geq1/2$ and $0<\gamma\leq T$. Nontrivial bounds for $N(\sigma,T)$ are called zero density estimates. The main result of this section is the following theorem, which will be instrumental in the proofs of Theorem \ref{umpires} and Theorem \ref{thm:multMain}.

\begin{theorem}
\label{thm:mainForS}
Let $k\in\N$, $0<\varepsilon\leq1/88$ and $x_0\geq e^{16}$. Assume that
\begin{equation}
\label{eq:szde}
N(\sigma,2T)-N(\sigma,T)\leq L\cdot T^{1-\frac{1}{4}\left(\sigma-\frac{1}{2}\right)}\log{T}
\end{equation}
for $T\geq T_0\geq \frac{1}{2}e^{1408}$, $\sigma\in[1/2,1]$ and $L>0$. Then we have
\begin{equation}
\label{eq:mainS}
\int_{T}^{2T}\left|S(t)+\frac{1}{\pi}\sum_{p\leq T^{\frac{3\varepsilon}{k}}}\frac{\sin{\left(t\log{p}\right)}}{\sqrt{p}}\right|^{2k}\dif{t} \leq C\left(\varepsilon,k,x_0\right)T
\end{equation}
for $T\geq \max\left\{x_0^{k/\varepsilon},2T_0\right\}$, where
\begin{gather}
C\left(\varepsilon,k,x_0\right)\de\widehat{C}\left(\varepsilon,\varepsilon,k,x_0\right), \label{eq:kk1} \\
\widehat{C}\left(\varepsilon_1,\varepsilon_2,k,x_0\right) \de \frac{1}{6}\left(1+\sum_{n=1}^{4}\widetilde{R}_n\left(\varepsilon_1,k,x_0\right)\right)\left(\frac{12a_2\left(x_0\right)a_4\left(x_0,\varepsilon_2, k\right)}{\varepsilon_2}k\right)^{2k}. \label{eq:Chat}
\end{gather}
Here,
\begin{gather*}
    \widetilde{R}_1\left(\varepsilon,k,x_0\right) \de \frac{8a_0 L}{\varepsilon}\left(8\varepsilon\right)^{2k}\frac{(2k)!}{k^{2k-1}} + \left(\frac{3\varepsilon}{2\pi a_2\left(x_0\right)a_4\left(x_0,\varepsilon, k\right)\sqrt{2k}}\right)^{2k}\left(13+\frac{1}{18^k}\right), \\
    \widetilde{R}_2\left(\varepsilon,k,x_0\right) \de \sqrt{13}\left(\frac{(12e)^2a_1\varepsilon}{a_2\left(x_0\right)a_4\left(x_0,\varepsilon, k\right)\sqrt{k}}\right)^{2k}\sqrt{1+\frac{8a_0 L}{\varepsilon}\left(\frac{8\varepsilon}{e}\right)^{8k}\frac{(8k)!}{k^{8k-1}}}, \\
    \widetilde{R}_3\left(\varepsilon,k,x_0\right) \de \sqrt{13}\left(\frac{24e^2\varepsilon}{\pi a_2\left(x_0\right)a_4\left(x_0,\varepsilon,k\right)\sqrt{k}}\right)^{2k}\sqrt{1+\frac{8a_0 L}{\varepsilon}\left(\frac{8\varepsilon}{e^2}\right)^{4k}\frac{(4k)!}{k^{4k-1}}}, \\
    \widetilde{R}_4\left(\varepsilon,k,x_0\right) \de \sqrt{13}\left(\frac{6a_1\varepsilon}{a_2\left(x_0\right)a_4\left(x_0,\varepsilon,k\right)\sqrt{k}}\right)^{2k}\sqrt{1+\frac{8a_0 L}{\varepsilon}\left(8\varepsilon\right)^{4k}\frac{(4k)!}{k^{4k-1}}},
\end{gather*}
and 
\begin{gather}
a_0\de 1.5453, \nonumber \\
a_1 \de 13 + \frac{26}{10\pi} + \frac{13}{3\pi e}, \label{eq:a1} \\
a_2\left(x_0\right) \de \frac{13}{2} + \frac{18}{10\pi} + \frac{13}{6\pi e} + \frac{2}{\log{x_0}}\left(52+\frac{124}{10\pi}+\frac{52}{3\pi e}\right), \label{eq:a2} \\
a_3\left(x_0\right) \de \frac{3\pi a_1}{2} + \frac{139}{75} + \frac{62\pi a_1}{75\log{x_0}}, \label{eq:a3} \\
a_4\left(x_0,\varepsilon,k\right) \de 1 + \frac{\varepsilon}{k}\left(\frac{\log{2}}{\log{x_0}} + \frac{a_3\left(x_0\right)}{\pi a_2\left(x_0\right)}\right).  \label{eq:a4} \nonumber
\end{gather}
For fixed $k$ and $x_0$, each function $\widetilde{R}_n\left(\varepsilon,k,x_0\right)$ is  strictly increasing in $\varepsilon$. For fixed $\varepsilon$ and $x_0$, each function $\widetilde{R}_n\left(\varepsilon,k,x_0\right) \to 0$ as $k \to \infty$. In particular, $\widehat{C}\left(\varepsilon_1,\varepsilon_2,k,x_0\right)\leq\widehat{C}\left(\varepsilon'_1,\varepsilon'_2,k,x_0\right)$ if $\varepsilon_1\leq\varepsilon'_1$ and $\varepsilon_2\geq\varepsilon'_2$, and when $k \to \infty$
\[
\widehat{C}\left(\varepsilon_1,\varepsilon_2,k,x_0\right)\sim\frac{1}{6}\left(\frac{12a_2\left(x_0\right)a_4\left(x_0,\varepsilon_2,k\right)k}{\varepsilon_2}\right)^{2k}.
\]
\end{theorem}

Selberg proved in \cite[Theorem 1]{SelbergContrib} that there exist $T_0>0$ and $L>0$ such that \eqref{eq:szde} is true for all $T\geq T_0$ and $\sigma\in[1/2,1]$. The next lemma provides an explicit version of his result, and improves Karatsuba and Korol\"{e}v's zero density estimate \cite[Theorem 1]{KaratsubaKorolevArgument} used in the proof of their version of Theorem \ref{thm:mainForS}.
\begin{lemma}
\label{lem:zde}
Let $T\geq \frac{1}{2}e^{1408}$ and $\sigma\in[1/2,1]$. Then the inequality \eqref{eq:szde} is true for $L\geq L_0\de 642.86$.
\end{lemma}

\begin{proof}
By \cite[Theorem 1]{SimonicEZDE} the inequality \eqref{eq:szde} with $L\geq L_0$ is true for $T\geq 10^{23.75}$ and $\sigma\in[1/2,0.646]$. Let $\sigma\geq0.646$. Using Ramar\'{e}'s zero density estimate \cite{Ram16} as corrected in \cite[Equation 1.5]{KLN} guarantees that
\[
N(\sigma,T)\leq 965(3T)^{\frac{8}{3}(1-\sigma)}\log^{5-2\sigma}{T}+51.5\log^2{T}.
\]
From this we get
\begin{flalign*}
N(\sigma,2T)-N(\sigma,T)&\leq \left(965\cdot6^{0.944}+1\right)T^{\frac{8}{3}(1-\sigma)}\log^{3.708}{(2T)} \\
&\leq 8.2\cdot L_0T^{\frac{8}{3}(1-\sigma)}\log^{4}{T}.
\end{flalign*}
Observe that $\left(58\sigma-37\right)/24>0.0195$. The proof is now complete since $8.2\log^{3}{T}\leq T^{0.0195}$ is true for $T\geq \frac{1}{2}e^{1408}$.
\end{proof}

Lemma \ref{lem:zde} asserts that \eqref{eq:mainS} is true for $T\geq x_0^{k/\varepsilon}$, where we can set  $L=L_0$ in functions $\widetilde{R}_i$. We note that $L$ could be improved, at the very least by using the new height to which RH has been verified \cite{PT20} in the work of \cite{KLN}. Since the main contribution in \eqref{eq:Chat} is due to \eqref{eq:a2}, which comes from Lemma \ref{lem:explSel}, improvements upon $L$ would not give significantly better estimates for \eqref{eq:mainS}, not even on RH, i.e., $L=0$.

We now consider integrals similar to those in \eqref{eq:mainS}, except that now we allow small translations in the variable $t$. Both results are needed in the proof of Theorem \ref{thm:difference}.

\begin{corollary}
\label{cor:mainForS}
Let $0\leq h\leq 1$. Then, in the notation of Theorem \ref{thm:mainForS}, we have
\begin{equation}\label{ale}
\int_{T}^{2T}\left|S(t+h)+\frac{1}{\pi}\sum_{p\leq T^{\frac{3\varepsilon}{k}}}\frac{\sin{\left((t+h)\log{p}\right)}}{\sqrt{p}}\right|^{2k}\dif{t} \leq C'\left(\varepsilon,k,x_0\right)T,
\end{equation}
where
\begin{equation}
\label{eq:kk2}
C'\left(\varepsilon,k,x_0\right) \de \left(1+x_0^{-\frac{k}{\varepsilon}}\right)\widehat{C}\left(\varepsilon,\frac{\varepsilon}{1+\frac{\varepsilon x_0^{-k/\varepsilon}}{k\log{x_0}}},k,x_0\right).
\end{equation}
\end{corollary}

\begin{proof}
The integral in (\ref{ale}) is not greater than
\[
\int_{T+h}^{2(T+h)}\left|S(t)+\frac{1}{\pi}\sum_{p\leq (T+h)^{\frac{3\varepsilon'}{k}}}\frac{\sin{\left(t\log{p}\right)}}{\sqrt{p}}\right|^{2k}\dif{t},
\]
where $\varepsilon'\de \varepsilon\log{T}/\log{(T+h)}$. By Theorem \ref{thm:mainForS}, this integral is bounded above by $(T+1)C\left(\varepsilon',k,x_0\right)$ for $0<\varepsilon'\leq1/88$ and $x_0\geq e^{16}$ since $T+h\geq \max\left\{x_0^{k/\varepsilon'},2T_0\right\}$. Observe that $\varepsilon\geq\varepsilon'\geq\varepsilon/\left(1+\varepsilon x_0^{-k/\varepsilon}/\left(k\log{x_0}\right)\right)$. Now the result follows from the last part of Theorem \ref{thm:mainForS}.
\end{proof}

Selberg investigated \eqref{eq:mainS} in \cite[Equation 5.3]{SelbergContrib}, for $\varepsilon=\left(a-1/2\right)/20$ where the integration goes from $T$ to $T+H$ with $T^a\leq H\leq T$ and $1/2<a\leq 1$. The motivation for this was to obtain lower bounds for a number of sign changes of $S(t)$ on a given interval, as well as to study $\Omega$-results for $S(t)$ and $S_1(t)$. Selberg did not provide the bound $O(H)$ uniformly in $k$. Ghosh in \cite[Lemma 5]{GhoshSign} carefully examined Selberg's proof and obtained an estimate $\ll(Ak)^{4k}H$. Tsang in his thesis \cite{TsangThesis} demonstrated that better bound $\ll(Ak)^{2k}H$ is possible, by refining \cite[Lemma 2]{GhoshSign}. In fact, the proof is the same as that of Lemma 12 in \cite{SelbergContrib}, see also Lemma \ref{lem:crucial} below.

Very little work was done in the explicit setting. Karatsuba and Korol\"{e}v provide in \cite[Lemma 7 on p.~440]{KaratKor} the estimate
\[
\int_{T}^{T+H} \left|S(t)+\frac{1}{\pi}\sum_{p< x}\frac{\sin{\left(t\log{p}\right)}}{\sqrt{p}}\right|^{2k}\dif{t} \leq \left(\frac{e^{37}k^2}{\varepsilon^{3}\pi^2}\right)^{k}H,
\]
where $0<\varepsilon<10^{-3}$, $H=T^{\frac{27}{82}+\varepsilon}$, $1\leq k<\left(\varepsilon\log{T}\right)/1920$, $T^{\varepsilon/(40k)}<x\leq T^{\varepsilon/(10k)}$, and $T\geq T_0$ for sufficiently large $T_0$. Summing up to $2T$, their result implies that we can take
\[
\left(1+\varepsilon\right)\left(\frac{e^{37}k^2}{(30\varepsilon)^{3}\pi^{2}}\right)^k
\]
instead of $C(\varepsilon,k)$ in Theorem \ref{thm:mainForS}, where $0<\varepsilon<\frac{1}{3}10^{-4}$ and $1\leq k<\left(\varepsilon\log{T}\right)/640$. For $k=1$ this bound is at least $10^{24}$, and for $k=2$ it is at least $10^{49}$.

We can observe that, for fixed $k$ and $x_0$, the function $C\left(\varepsilon,k,x_0\right)$ decreases in $\varepsilon$. This means that we would obtain the best bounds by choosing $\varepsilon=1/88$.
We also choose $x_0$ sufficiently large such that the resulting bounds are close to their limiting values. This is not an issue because for Corollary \ref{cor:mainForS} we need only bounds for sufficiently large $T$. Observe also that our values are much smaller than what Karatsuba--Korol\"{e}v's estimates give.

\begin{corollary}
\label{cor:mainValues1}
Let $\varepsilon=1/88$ and $x_0=e^{2\cdot 10^4}$. Then we have $C\left(\varepsilon,1,x_0\right)<1.44161\cdot10^{11}$ and $C\left(\varepsilon,2,x_0\right)<2.69927\cdot10^{21}$, with the same values also for $C'$. Define $\omega_0\de e^{12.8471}$. For all positive integers $k$ we have $C\left(\varepsilon,k,x_0\right)\leq\left(\omega_0k\right)^{2k}$, with the same bound also for $C'$.
\end{corollary}

Before proceeding to the proof of Theorem \ref{thm:mainForS}, we need some preliminary definitions and results. Let $2\leq x\leq t^2$ and $n$ be a positive integer. Define
\[
\sigma_{x,t}\de \frac{1}{2} + 2\mathop{\max}_{\substack{\beta\\|t-\gamma|\leq x^{\frac{3\left|\beta-\frac{1}{2}\right|}{\log{x}}}}}\left\{\left|\beta-\frac{1}{2}\right|,\frac{1}{\log{x}}\right\},
\]
and
\[
\Lambda_x(n) \de \left\{
\begin{array}{ll}
    \Lambda(n), & 1\leq n\leq x, \\
    \Lambda(n)\frac{\log^2{\left(x^3/n\right)}-2\log^2{\left(x^2/n\right)}}{2\log^2{x}}, & x< n\leq x^2, \\
    \Lambda(n)\frac{\log^2{\left(x^3/n\right)}}{2\log^2{x}}, & x^2<n\leq x^3.
\end{array}
\right.
\]
Observe that $\Lambda_x(n)\leq\Lambda(n)$. Define also
\[
r(x,t) \de \sum_{n\leq x^3}\frac{\Lambda_x(n)}{n^{\sigma_{x,t}+\ie t}}.
\]
The following lemma is an explicit version of Selberg's approximation formula for $S(t)$, see \cite[Theorem 2]{SelbergContrib}.

\begin{lemma}
\label{lem:explSel}
Let $e^{16}\leq x_0\leq x\leq t^2$. Then
\[
\left|S(t)+\frac{1}{\pi}\sum_{n\leq x^3}\frac{\Lambda_x(n)\sin{\left(t\log{n}\right)}}{n^{\sigma_{x,t}}\log{n}}\right| \leq \left(\sigma_{x,t}-\frac{1}{2}\right)\left(a_1\left|r(x,t)\right|+a_2\left(x_0\right)\log{|t|}\right),
\]
where $a_1$ and $a_2\left(x_0\right)$ are defined by \eqref{eq:a1} and \eqref{eq:a2}, respectively.
\end{lemma}

\begin{proof}
This follows from the proof of Theorem 1 in \cite{KaratsubaKorolevArgument}.
\end{proof}

Setting $x_0=e^{16}$ in Lemma \ref{lem:explSel} provides bounds $a_1<15$ and $a_2\left(x_0\right)<15$. With the value 15, this is original statement of Theorem 1 in \cite{KaratsubaKorolevArgument}, but our formulation provides slightly better bounds in Corollary \ref{cor:mainValues1}.

\begin{lemma}
\label{lem:crucial}
Let $x\geq2$, $T\geq2T_0\geq e^{1408}$, $1\leq\xi\leq x^{4k}$ and $x^3\xi^2\leq T^{1/8}$. Then
\begin{equation}\label{stout}
\int_{T}^{2T} \left(\sigma_{x,t}-\frac{1}{2}\right)^{\nu}\xi^{\sigma_{x,t}-\frac{1}{2}}\dif{t} \leq \mathcal{C}\left(x,T,\xi,\nu\right)\frac{T}{\log^{\nu}{x}},
\end{equation}
with
\begin{equation}
\label{eq:C4}
\mathcal{C}\left(x,T,\xi,\nu\right) \de 2^{\nu}\xi^{\frac{2}{\log{x}}}+La_0 2^{4\nu+3}\nu!\left(\frac{\log{x}}{\log{T}}\right)^{\nu}\frac{\log{T}}{\log{x}},
\end{equation}
where $L$ and $a_0$ are from Theorem \ref{thm:mainForS}, and $\nu$ is a positive integer.
\end{lemma}

\begin{proof}
Let $\mathscr{A}_1\subseteq(T,2T]$ be a set of those $t$ for which $|\beta-1/2|\leq 1/\log{x}$ for all the nontrivial zeroes $\rho=\beta+\ie\gamma$ that satisfy $|t-\gamma|\leq x^{3|\beta-1/2|}/\log{x}$. Define $\mathscr{A}_2\de (T,2T]\setminus\mathscr{A}_1$. Denote by $\mathcal{I}_1$ and $\mathcal{I}_2$ the integration of the integrand in (\ref{stout}) through the sets $\mathscr{A}_1$ and $\mathscr{A}_2$, respectively, whence the integral in (\ref{stout}) is $\mathcal{I}_1+\mathcal{I}_2$.

If $t\in\mathscr{A}_1$, then $\sigma_{x,t}-1/2=2/\log{x}$. This implies that
\[
\mathcal{I}_1 \leq 2^{\nu}\xi^{\frac{2}{\log{x}}}\frac{T}{\log^{\nu}{x}}.
\]

If $t\in\mathscr{A}_2$, then there exists a zero $\rho$ with $|t-\gamma|\leq x^{3|\beta-1/2|}/\log{x}$ and
\[
\left(\sigma_{x,t}-\frac{1}{2}\right)^{\nu}\xi^{\sigma_{x,t}-\frac{1}{2}}=2^{\nu}|\beta-1/2|^{\nu}\xi^{2|\beta-1/2|}.
\]
Since $x^{3|\beta-1/2|}\leq x^{3/2}\leq T^{1/16}$, it means that $\gamma\in\left(a,b\right]$, where $a\de T-T^{1/16}$ and $b\de 2T+T^{1/16}$. Therefore, we obtain
\begin{flalign*}
\mathcal{I}_2 &\leq \sum_{a<\gamma\leq b}\int_{\gamma-\frac{x^{3\left|\beta-\frac{1}{2}\right|}}{\log{x}}}^{\gamma+\frac{x^{3\left|\beta-\frac{1}{2}\right|}}{\log{x}}} 2^{\nu}\left|\beta-\frac{1}{2}\right|^{\nu}\xi^{2\left|\beta-\frac{1}{2}\right|}\dif{t} \leq \frac{2^{\nu+2}}{\log{x}} \mathop{\sum}_{\substack{a<\gamma\leq b \\\beta>1/2}} \left(\beta-\frac{1}{2}\right)^{\nu}\left(x^3\xi^2\right)^{\beta-\frac{1}{2}} \\
&\leq \frac{2^{\nu+2}}{8\log{x}}\int_{\frac{1}{2}}^{1}\left(\left(\sigma-\frac{1}{2}\right)^{\nu}\log{T}+8\nu\left(\sigma-\frac{1}{2}\right)^{\nu-1}\right)\times \\
&\times\left(N(\sigma,b)-N(\sigma,a)\right)T^{\frac{1}{8}\left(\sigma-\frac{1}{2}\right)}\dif{\sigma},
\end{flalign*}
where the reader is advised to consult \cite[p.~438]{KaratKor} for details on the derivation of the third inequality. Further,
\[
N(\sigma,b)-N(\sigma,a) \leq \left(N\left(\sigma,4\alpha_1 T\right)-N\left(\sigma,2\alpha_1 T\right)\right) + \left(N\left(\sigma, 2\alpha_1 T\right)-N\left(\sigma,\alpha_1 T\right)\right),
\]
where $\alpha_1\de\frac{1}{2}+\frac{1}{4}e^{-1320}$. Because $\alpha_1T\geq T_0\geq\frac{1}{2}e^{1408}$, we have in the notation of Theorem \ref{thm:mainForS} the following estimate
\begin{flalign*}
N(\sigma,b)-N(\sigma,a) &\leq L\left(\alpha_1^{\frac{7}{8}}+2\alpha_1\right)\left(1+\frac{\log{(2\alpha_1)}}{1408}\right)T^{1-\frac{1}{4}\left(\sigma-\frac{1}{2}\right)}\log{T} \\
&\leq La_0 T^{1-\frac{1}{4}\left(\sigma-\frac{1}{2}\right)}\log{T},
\end{flalign*}
valid for $\sigma\in[1/2,1]$. Applying this bound to the last inequality for $\mathcal{I}_2$, making a change of variable $u\mapsto\sigma-1/2$ and then integrating, we obtain
\[
\mathcal{I}_2 \leq La_0\frac{2^{\nu+2}}{8\log{x}}T\log{T}\int_{0}^{\infty}\frac{u^{\nu}\log{T}+8\nu u^{\nu-1}}{T^{\frac{1}{8}u}}\dif{u} = La_0\frac{\log{T}}{\log{x}}2^{4\nu+3}\nu!\frac{T}{\log^{\nu}{T}}.
\]
Now the result easily follows.
\end{proof}

Let $e^{16}\leq x_0\leq x\leq t^2$. By Lemma \ref{lem:explSel} we have
\[
\left|\pi S(t)+\sum_{p\leq x^3}\frac{\sin{\left(t\log{p}\right)}}{\sqrt{p}}\right| \leq a_2\left(x_0\right)\pi\left(\sigma_{x,t}-\frac{1}{2}\right)\log{|t|} + \sum_{n=1}^6 \left|A_n\right|,
\]
where
\begin{gather*}
    A_1 \de \sum_{p\leq x^3}\frac{\Lambda(p)-\Lambda_x(p)}{\sqrt{p}\log{p}}\sin{\left(t\log{p}\right)}, \\
    A_2 \de \sum_{p\leq x^3} \frac{\Lambda_x(p)}{\sqrt{p}\log{p}}\left(1-p^{\frac{1}{2}-\sigma_{x,t}}\right)\sin{\left(t\log{p}\right)}, \\
    A_3 \de \frac{1}{2}\sum_{p\leq x^{3/2}} \frac{\Lambda_x\left(p^2\right)}{p\log{p}}\left(1-p^{1-2\sigma_{x,t}}\right)\sin{\left(2t\log{p}\right)},
\end{gather*}
\begin{gather*}
    A_4 \de -\frac{1}{2}\sum_{p\leq x^{3/2}} \frac{\Lambda_x\left(p^2\right)}{p\log{p}}\sin{\left(2t\log{p}\right)}, \\
    A_5 \de -\sum_{r\geq3}\sum_{p^r<x^3} \frac{\Lambda_x\left(p^r\right)}{p^{r\sigma_{x,t}}r\log{p}}\sin{\left(rt\log{p}\right)}, \\
    A_6 \de a_1\pi\left(\sigma_{x,t}-\frac{1}{2}\right)\left|\sum_{n\leq x^3}\frac{\Lambda_x(n)}{n^{\sigma_{x,t}+\ie t}}\right|.
\end{gather*}
We trivially have
\begin{equation}
\label{eq:R1}
\left|A_1\right| \leq R_1(x,t)\de \left|\sum_{p\leq x^3}\frac{\Lambda(p)-\Lambda_x(p)}{\sqrt{p}\log{p}}p^{\ie t}\right|.
\end{equation}
Next,
\begin{flalign}
\left|A_2\right| &\leq \left|\sum_{p\leq x^3} \frac{\Lambda_x(p)}{\sqrt{p}\log{p}}\left(1-p^{\frac{1}{2}-\sigma_{x,t}}\right)p^{\ie t}\right| \nonumber \\
&\leq \int_{\frac{1}{2}}^{\sigma_{x,t}}\left|\sum_{p\leq x^3}\frac{\Lambda_x(p)}{p^{\sigma'-\ie t}}\right|\dif{\sigma'} = \int_{\frac{1}{2}}^{\sigma_{x,t}}\left|x^{\sigma'-\frac{1}{2}}\int_{\sigma'}^{\infty}x^{\frac{1}{2}-\sigma}\sum_{p\leq x^3}\frac{\Lambda_x(p)\log{(xp)}}{p^{\sigma-\ie t}}\dif{\sigma}\right|\dif{\sigma'} \nonumber \\
&\leq \int_{\frac{1}{2}}^{\sigma_{x,t}}x^{\sigma'-\frac{1}{2}}\int_{\sigma'}^{\infty}x^{\frac{1}{2}-\sigma}\left|\sum_{p\leq x^3}\frac{\Lambda_x(p)\log{(xp)}}{p^{\sigma-\ie t}}\right|\dif{\sigma}\dif{\sigma'} \nonumber \\
&\leq \left(\sigma_{x,t}-\frac{1}{2}\right)x^{\sigma_{x,t}-\frac{1}{2}}\int_{\frac{1}{2}}^{\infty} x^{\frac{1}{2}-\sigma}\left|\sum_{p\leq x^3}\frac{\Lambda_x(p)\log{(xp)}}{p^{\sigma-\ie t}}\right|\dif{\sigma}\ed R_2(x,t). \label{eq:R2}
\end{flalign}
Also
\[
\left|A_3\right|\leq \frac{1}{2}\sum_{p\leq x^{3/2}} \frac{1-p^{1-2\sigma_{x,t}}}{p} \leq \frac{3}{2}\left(\sigma_{x,t}-\frac{1}{2}\right)\log{x},
\]
where we used $1-p^{1-2\sigma_{x,t}}\leq 2\left(\sigma_{x,t}-1/2\right)\log{p}$, and
\begin{equation}\label{desk}
\sum_{p\leq X}\frac{\log{p}}{p}\leq \log{X}
\end{equation}
for $X>1$, see \cite[Equation 3.24]{RosserSchoenfeld}. Next,
\begin{equation}
\label{eq:R4}
\left|A_4\right|\leq R_4(x,t)\de \frac{1}{2}\left|\sum_{p\leq x^{3/2}} \frac{\Lambda_x\left(p^2\right)}{p\log{p}}p^{2\ie t}\right|,
\end{equation}
and
\[
\left|A_5\right|\leq \sum_{r\geq3}\sum_{p^r<x^3} \frac{1}{rp^{r\sigma_{x,t}}} \leq \sum_{r\geq3}\sum_{p} \frac{1}{rp^{r/2}} \leq \frac{1}{3}\sum_{p}\frac{1}{p\left(\sqrt{p}-1\right)}\leq \frac{2.12}{3}.
\]
Because
\begin{multline*}
\left|\sum_{n\leq x^3}\frac{\Lambda_x(n)}{n^{\sigma_{x,t}+\ie t}}\right| \leq \left|\sum_{p\leq x^3}\frac{\Lambda_x(p)}{\sqrt{p}}\left(1-p^{\frac{1}{2}-\sigma_{x,t}}\right)p^{\ie t}\right| + \left|\sum_{p\leq x^3}\frac{\Lambda_{x}(p)}{\sqrt{p}}p^{\ie t}\right| \\
+ \left|\sum_{p\leq x^{3/2}}\frac{\Lambda_x\left(p^2\right)}{p^{2\sigma_{x,t}+2\ie t}}\right| + \frac{1}{3}\sum_{p}\frac{\log{p}}{p\left(\sqrt{p}-1\right)},
\end{multline*}
we have
\[
\left|A_6\right|\leq R_3 + R_5 + \frac{3a_1\pi}{2}\left(\sigma_{x,t}-\frac{1}{2}\right)\log{x} + \frac{2.48a_1\pi}{3}\left(\sigma_{x,t}-\frac{1}{2}\right),
\]
where
\begin{gather}
R_3(x,t)\de a_1\pi\left(\sigma_{x,t}-\frac{1}{2}\right)^2x^{\sigma_{x,t}-\frac{1}{2}}\int_{\frac{1}{2}}^{\infty} x^{\frac{1}{2}-\sigma}\left|\sum_{p\leq x^3}\frac{\Lambda_x(p)\log{(xp)}\log{p}}{p^{\sigma-\ie t}}\right|\dif{\sigma}, \label{eq:R3} \\
R_5(x,t)\de a_1\pi\left(\sigma_{x,t}-\frac{1}{2}\right)\left|\sum_{p\leq x^3}\frac{\Lambda_{x}(p)}{\sqrt{p}}p^{\ie t}\right|. \label{eq:R5}
\end{gather}
In conclusion,
\[
\left|\pi S(t)+\sum_{p\leq x^3}\frac{\sin{\left(t\log{p}\right)}}{\sqrt{p}}\right| \leq \left(a_3+a_2\pi\frac{\log{|t|}}{\log{x}}\right)\left(\sigma_{x,t}-\frac{1}{2}\right)\log{x} + \sum_{n=1}^{5}R_n.
\]
Applying the inequality
\begin{equation}
\label{eq:corHolder}
\left(\sum_{n=1}^{m}u_n\right)^{l} \leq m^{l-1}\sum_{n=1}^{m}u_n^l,
\end{equation}
which simply follows from H\"{o}lder's inequality, finally gives the following.

\begin{proposition}
\label{prop:S}
Let $k\in\N$ and $e^{16}\leq x_0\leq x\leq T^2$. Assume that there exists $b>0$ such that $\log{(2T)}\leq b\log{x}$. Then we have
\begin{flalign}
\label{eq:moment}
\int_{T}^{2T}\left|S(t)+\frac{1}{\pi}\sum_{p\leq x^3}\frac{\sin{\left(t\log{p}\right)}}{\sqrt{p}}\right|^{2k}\dif{t} &\leq \frac{1}{6}\left(6a_2b\left(1+\frac{a_3}{\pi a_2b}\right)\log{x}\right)^{2k}\times \nonumber \\
&\times\int_{T}^{2T}\left(\sigma_{x,t}-\frac{1}{2}\right)^{2k}\dif{t} \nonumber \\
&+ \frac{1}{6}\left(\frac{6}{\pi}\right)^{2k}\sum_{n=1}^{5}\int_{T}^{2T}R_n^{2k}(x,t)\dif{t},
\end{flalign}
where  $R_1(x,t),\ldots,R_5(x,t)$ are given by \eqref{eq:R1}, \eqref{eq:R2}, \eqref{eq:R3}, \eqref{eq:R4}, and \eqref{eq:R5}, respectively, and $a_2\left(x_0\right)$ and $a_3\left(x_0\right)$ are given by \eqref{eq:a2} and \eqref{eq:a3}, respectively.
\end{proposition}

By Lemma \ref{lem:crucial}, the first term on the right-hand side of \eqref{eq:moment} is not greater than
\begin{equation}
\label{eq:Main1}
\frac{1}{6}\left(6a_2b\left(1+\frac{a_3}{\pi a_2b}\right)\right)^{2k} \mathcal{C}\left(x,T,1,2k\right)\cdot T,
\end{equation}
where $T\geq 2T_0\geq e^{1408}$, $e^{16}\leq x_0\leq x\leq T^{1/24}$, and the function $\mathcal{C}$ is defined by \eqref{eq:C4}. We are now ready to estimate the remaining integrals in \eqref{eq:moment}.

\subsection{Estimate for the integrals $R_i(x,t)$}

We are now ready to estimate the remaining integrals in \eqref{eq:moment}.

\subsubsection{Estimate for the integrals $R_1(x,t)$}
We have
\[
\int_{T}^{2T} R_1^{2k}(x,t) \dif{t} = \int_{0}^{T}\left|\sum_{p\leq x^3}\frac{a_5(p)}{\sqrt{p}}p^{\ie u}\right|^{2k} \dif{u},
\]
where
\[
a_5(p) \de \left(1-\frac{\Lambda_x(p)}{\log{p}}\right)p^{\ie T}.
\]
If $p\leq x$, then $a_5(p)=0$. Let $x<p\leq x^2$. By definition of $\Lambda_x(n)$ we then have
\[
\frac{\Lambda_x(p)}{\log{p}} = 1-\frac{1}{2}\left(\frac{\log{p}}{\log{x}}-1\right)^2,
\]
which implies that $\left|a_5(p)\right|\leq (3/2)\log{p}/\log{x^3}$. Take $x^2<p\leq x^3$. Then
\[
\frac{\Lambda_x(p)}{\log{p}} = \frac{1}{2}\left(3-\frac{\log{p}}{\log{x}}\right)^2,
\]
which also implies $\left|a_5(p)\right|\leq (3/2)\log{p}/\log{x^3}$. Thus we can use \cite[Lemma 3]{KaratKor} to obtain
\begin{equation}
\label{eq:Main2}
\frac{1}{T}\int_{T}^{2T} R_1^{2k}(x,t) \dif{t} \leq 13\left(\frac{9}{2}k\right)^k,
\end{equation}
where
\begin{equation*}
x=T^{\frac{c_1}{3k}}, \quad T\geq \exp{\left(\max\left\{e^2k/c_1,e^3/c_1\right\}\right)}, \quad 0<c_1\leq 1.
\end{equation*}

\subsubsection{Estimate for the integrals with $R_2(x,t)$ and $R_3(x,t)$}

Define
\begin{multline*}
I(n) \de \left(a_1\pi\right)^{2nk} \int_{T}^{2T} \left(\sigma_{x,t}-\frac{1}{2}\right)^{2k(n+1)}x^{2k\left(\sigma_{x,t}-\frac{1}{2}\right)}\times \\
\times\left(\int_{\frac{1}{2}}^{\infty}x^{\frac{1}{2}-\sigma}\left|\sum_{p\leq x^3}\frac{\Lambda_x(p)\log{(xp)}\log^n{p}}{p^{\sigma-\ie t}}\right|\dif{\sigma}\right)^{2k}\dif{t}.
\end{multline*}
We need to obtain bounds on $I(0)$ and $I(1)$. By the Cauchy--Schwarz inequality we have $I(n)\leq \left(a_1\pi\right)^{2nk}\sqrt{j_1(n)}\sqrt{j_2(n)}$, where
\begin{gather*}
    j_1(n) \de \int_{T}^{2T} \left(\sigma_{x,t}-\frac{1}{2}\right)^{4k(n+1)}x^{4k\left(\sigma_{x,t}-\frac{1}{2}\right)}\dif{t}, \\
    j_2(n) \de \int_{T}^{2T} \left(\int_{\frac{1}{2}}^{\infty}x^{\frac{1}{2}-\sigma}\left|\sum_{p\leq x^3}\frac{\Lambda_x(p)\log{(xp)}\log^n{p}}{p^{\sigma-\ie t}}\right|\dif{\sigma}\right)^{2k}\dif{t}.
\end{gather*}
H\"{o}lder's inequality with $p=4k/(4k-1)$ and $q=4k$, and Fubini's theorem imply that
\[
j_2(n) \leq \left(\log{x}\right)^{4k(1+n)+1} \int_{0}^{\infty} x^{-u}\int_{0}^{T}\left|\sum_{p\leq x^3}\frac{a_6(p)}{\sqrt{p}}p^{\ie v}\right|^{4k}\dif{v}\dif{u},
\]
where
\[
a_6(p) \de \frac{\Lambda_x(p)\log{(xp)}\log^n{p}}{p^{u}\log^{2+n}{x}}p^{\ie T}.
\]
Because $\Lambda_x(p)\leq\log{p}$, we have $\left|a_6(p)\right|\leq 12\cdot 3^n \log{p}/\log{x^3}$. Then \cite[Lemma 3]{KaratKor} implies
\begin{equation}\label{porter}
j_2(n) \leq 13\left(576\cdot 3^{2n}k\right)^{2k}\left(\log{x}\right)^{4k(n+1)}T,
\end{equation}
where
\begin{equation}
\label{eq:cond2}
x=T^{\frac{c_2}{6k}}, \quad T\geq \exp{\left(\max\left\{2e^2k/c_2,e^3/c_2\right\}\right)}, \quad 0<c_2\leq 1.
\end{equation}
By Lemma \ref{lem:crucial} we have
\begin{equation}\label{ipa}
j_1(n) \leq \mathcal{C}\left(x,T,x^{4k},4k(n+1)\right)\left(\log{x}\right)^{-4k(n+1)}T,
\end{equation}
where
\begin{equation}
\label{eq:cond3}
2\leq x\leq T^{\frac{1}{8(8k+3)}}, \quad T\geq \max\left\{2T_0,2^{8(8k+3)}\right\}, \quad T_0\geq \frac{1}{2}e^{1408}.
\end{equation}
The bounds for $j_1(n)$ and $j_2(n)$ in (\ref{porter}) and (\ref{ipa}) give us
\begin{multline}
\label{eq:Main3}
\frac{1}{T}\int_{T}^{2T} R_2^{2k}(x,t)+R_3^{2k}(x,t) \dif{t} \leq \sqrt{13}\left(24\sqrt{k}\right)^{2k}\sqrt{\mathcal{C}\left(x,T,x^{4k},4k\right)} \\
+\sqrt{13}\left(72\pi a_1\sqrt{k}\right)^{2k}\sqrt{\mathcal{C}\left(x,T,x^{4k},8k\right)},
\end{multline}
where $x$ and $T$ are subject to \eqref{eq:cond2} and \eqref{eq:cond3}.

\subsubsection{Estimate for the integral with $R_4(x,t)$}

We have
\begin{equation}
\label{eq:Main4}
\int_{T}^{2T} R_4^{2k}(x,t)\dif{t} = 2^{-2k}\int_{0}^{T}\left|\sum_{p\leq x^{3/2}}\frac{a_7(p)}{p}p^{2\ie t}\right|^{2k} \dif{t} \leq \left(\frac{1}{4}k\right)^k T,
\end{equation}
where $a_7(p)\de \Lambda_x\left(p^2\right)p^{2\ie T}/\log{p}$, since $\left|a_7(p)\right|\leq 1$ and then the last inequality follows by \cite[Lemma 3]{KaratKor}. Inequality \eqref{eq:Main4} is valid for
\begin{equation*}
x=T^{\frac{2c_3}{3k}}, \quad T\geq \exp{\left(\max\left\{e^2k/c_3,e^3/c_3\right\}\right)}, \quad 0<c_3\leq 1.
\end{equation*}

\subsubsection{Estimate for the integral with $R_5(x,t)$}

By the Cauchy--Schwarz inequality we have
\[
\int_{T}^{2T} R_5^{2k}(x,t) \dif{t} \leq \left(a_1\pi\log{x}\right)^{2k}\left|\int_{T}^{2T}\left(\sigma_{x,t}-\frac{1}{2}\right)^{4k}\dif{t}\int_{0}^{T}\left|\sum_{p\leq x^3}\frac{a_8(p)}{\sqrt{p}}p^{\ie t}\right|^{4k}\dif{t}\right|^{\frac{1}{2}},
\]
where $a_8(p)\de \Lambda_x(p)p^{\ie T}/\log{x}$. By Lemma \ref{lem:crucial} and \cite[Lemma 3]{KaratKor} we obtain
\begin{equation}
\label{eq:Main5}
\frac{1}{T}\int_{T}^{2T} R_5^{2k}(x,t)\dif{t} \leq \sqrt{13}\left(6\pi a_1 \sqrt{k}\right)^{2k}\sqrt{\mathcal{C}\left(x,T,1,4k\right)},
\end{equation}
valid for $x$ and $T$ satisfying $2\leq x\leq T^{1/24}$, $T\geq 2T_0\geq e^{1408}$ and \eqref{eq:cond2}.

\subsection{Proof of Theorem \ref{thm:mainForS}.} 

Let $x=T^{\varepsilon/k}$, $x_0\geq e^{16}$ and $T_0\geq \frac{1}{2}e^{1408}$. Combining all restrictions on $x$ and $T$ previously obtained in the above sections, we have
\begin{gather*}
\frac{k\log{x_0}}{\log{T}} \leq \varepsilon \leq \frac{k}{8(8k+3)}, \\
T\geq \exp{\left(\max\left\{\log{\left(2T_0\right)},\frac{k}{\varepsilon}\log{x_0},\frac{2e^2k}{3\varepsilon},\frac{2e^3}{3\varepsilon},8(8k+3)\log{2}\right\}\right)}.
\end{gather*}
From the first condition we get $\varepsilon\leq1/88$, which implies $k\varepsilon^{-1}\log{x_0}\geq1408k$ and consequently $T\geq\max\left\{x_0^{k/\varepsilon},2T_0\right\}$.
Thus means that the conditions in Theorem \ref{thm:mainForS} are satisfied. Define
\[
b\de \frac{k}{\varepsilon}\left(1+\frac{\varepsilon\log{2}}{k\log{x_0}}\right).
\]
Such $b$ satisfies the condition in Proposition \ref{prop:S}. We observe that the largest contribution to the final bound comes from \eqref{eq:Main1}. Therefore, the term $1+\widetilde{R}_1$ in \eqref{eq:Chat} comes from \eqref{eq:Main1}, \eqref{eq:Main2} and \eqref{eq:Main4}, while the remainder comes from \eqref{eq:Main3} and \eqref{eq:Main5}.

\section{Explicit second and fourth power moments of $S(t+h)-S(t)$ and the proof of Theorem \ref{umpires}}
\label{stumps}
The main goal of this section is to prove Theorem \ref{umpires}, which we will see is a consequence of the following, explicit version of Theorem \ref{thm:tsang} for $k\in\{1,2\}$ and $H=T$. We note that the role of $\alpha$ in what follows is not of much importance and shows that the choices $\alpha=3$ and $\alpha=2$ by Fujii and Tsang, respectively, are arbitrary and do not come from the method of the proof.

\begin{theorem}
\label{thm:difference}
Let $0<h\leq 1$, $0<\varepsilon\leq 3/88$ and $1\leq\alpha\leq 7/(10\varepsilon)$. For $T\geq T_{1}\geq\max\left\{x_0^{6/\varepsilon},2T_0\right\}$, where $x_0$ and $T_0$ are from Theorem \ref{thm:mainForS}, we have
\begin{gather*}
\left|\mathcal{J}_2(T,h)-\frac{T}{\pi^2}\log{\left(\alpha+h\log{T}\right)}\right|
\leq C_2T\log^{\frac{1}{2}}{\left(\alpha+h\log{T}\right)}, \\
\left|\mathcal{J}_4(T,h)-\frac{3T}{\pi^4}\log^2{\left(\alpha+h\log{T}\right)}\right|
\leq C_3T\log^{\frac{3}{2}}{\left(\alpha+h\log{T}\right)},
\end{gather*}
where
\begin{gather*}
    C_2 \de 2\sqrt{E_1E_2} + \frac{E_1+D_1}{\log^{\frac{1}{2}}{\left(\alpha+h\log{T}\right)}}, \\
    C_3 \de 4\sqrt[4]{F_1^3F_2} + \frac{6\sqrt{F_1F_2}+D_2}{\log^{\frac{1}{2}}{\left(\alpha+h\log{T}\right)}} + \frac{4\sqrt[4]{F_1F_2^3}}{\log{\left(\alpha+h\log{T}\right)}} + \frac{F_2}{\log^{\frac{3}{2}}{\left(\alpha+h\log{T}\right)}},
\end{gather*}
with
\begin{gather*}
    E_1 \de \frac{1}{\pi^2}+\frac{D_1}{\log{\left(\alpha+h\log{T}\right)}}, \quad
    F_1 \de \frac{3}{\pi^4} + \frac{D_2}{\log{\left(\alpha+h\log{T}\right)}}, \\
    E_2 \de 2\left(C\left(\varepsilon/3,1,x_0\right)+C'\left(\varepsilon/3,1,x_0\right)\right), \quad
    F_2 \de 8\left(C\left(\varepsilon/3,2,x_0\right)+C'\left(\varepsilon/3,2,x_0\right)\right),
\end{gather*}
and functions $D_1$, $D_2$, $C$ and $C'$ are defined by \eqref{eq:d1}, \eqref{eq:d2}, \eqref{eq:kk1} and \eqref{eq:kk2}, respectively.
\end{theorem}

We will see that Theorem \ref{thm:mainForS} is a key insert in the proof of Theorem \ref{thm:difference}. Assuming for the moment that Theorem \ref{thm:difference} holds, we can quickly prove Theorem \ref{umpires}.

\subsection{Proof of Theorem \ref{umpires}}
Take $\alpha=1$, $\varepsilon=3/88$ and $h=2\pi M\lambda/\log{T}$. Using the trivial bound $\log{(2\pi M)}\leq\log{\left(1+2\pi M\lambda\right)}$, Theorem \ref{thm:difference} asserts that there exists $T_0>0$ such that
\begin{equation}\label{lager}
 \mathcal{J}_2\left(T,\frac{2\pi M\lambda}{\log{T}}\right)\geq \frac{\log{(2\pi M)}}{\pi^2}\left(1-\frac{\pi^2C_2}{\sqrt{\log{(2\pi M)}}}\right)T
\end{equation}
for all $T\geq T_0$.
We seek the smallest value of $M$ such that the lower bound in (\ref{lager}) is positive. Choosing $\log(\log(2\pi M)) = 30.76$ works, but the slightly smaller value of $30.75$ fails. This gives the value of $M_{2}$. To estimate $M_{3}$ we use the inequality $\log(1+x) \leq x$, and that fact that $\lambda \leq 2$ to deduce that
$$\log(1 + 2\pi M \lambda) \leq \log (2\pi M) + 0.7.$$
This could ultimately be improved with our final choice of $\lambda$ but this makes no difference to the calculations. We therefore obtain $M_{3}$.
Finally, the choice of $M_{1}$ follows from H\"{o}lder's inequality, since then we can take $M_1\de\sqrt{M_2^3/M_3}$.

%
%

\subsection{Preliminary results}It remains to prove Theorem \ref{thm:difference}. The next two lemmas are explicit versions of Lemmas 1 and 2 of \cite{Tsang}. Note that the first lemma is a generalization of Preissmann's refinement of 
 Hilbert's inequality, namely
\begin{equation}
\label{eq:MVH}
\left|
\int_{T_1}^{T_2}\left|\sum_{n\leq X}a_n n^{-\ie t}\right|^2\dif{t} - \left(T_2-T_1\right)\sum_{n\leq X}\left|a_n\right|^2
\right| \leq 3\pi m_0\sum_{n\leq X}n\left|a_n\right|^2,
\end{equation}
where $X\geq2$, $T_1$ and $T_2$ are real numbers, $\left\{a_n\right\}_{n\leq X}$ is a sequence of complex numbers, and $m_0\de \sqrt{1+\frac{2}{3}\sqrt{\frac{6}{5}}}$. Inequality \eqref{eq:MVH} easily follows from \cite{Preissmann} after observing that $\left|\log{(n/m)}\right|\geq 2/(3n)$ for positive and distinct integers $n$ and $m$.

\begin{lemma}
\label{lem:MVP}
Let $X\geq 2$, and $\left\{a_n\right\}_{n\leq X}$ and $\left\{b_n\right\}_{n\leq X}$ be two sequences of complex numbers. Define
\[
f(t)\de \sum_{n\leq X}a_n n^{-\ie t}, \quad g(t)\de \sum_{n\leq X}b_n n^{-\ie t}.
\]
Then
\[
\left|\int_{T_1}^{T_2} f(t)\overline{g(t)}\dif{t} - \left(T_2-T_1\right)\sum_{n\leq X}a_n\overline{b_n}\right| \leq 3\pi m_0\sqrt{\sum_{n\leq X} n|a_n|^2}\sqrt{\sum_{n\leq X} n|b_n|^2},
\]
where $T_1$ and $T_2$ are real numbers.
\end{lemma}

\begin{proof}
The proof is the same as in \cite[p.~373]{Tsang}, except that we used \eqref{eq:MVH} at the appropriate places in Tsang's proof.
\end{proof}

\begin{lemma}
\label{lem:tsang2}
Let $X\geq 2$ and $\left\{a_p\right\}_{p\leq X}$ be a sequence of complex numbers with prime indices $p$, and $k$ a positive integer. Then
\[
\left|\int_{T_1}^{T_2} \Im\left\{\sum_{p\leq X} \frac{a_p}{p^{\ie t}}\right\}^{2k}\dif{t} - \frac{T_2-T_1}{4^k}\binom{2k}{k}\sum_{\mathbf{p}}\left|a_{\mathbf{p}}\right|^2\left|\mathbf{p}\right|\right| \leq C_1(k)\left(\sum_{p\leq X}p\left|a_p\right|^2\right)^k,
\]
where
\[
C_1(k) \de \frac{3\pi m_0}{4^k}\sum_{m=0}^{2k}\binom{2k}{m}\sqrt{m!(2k-m)!},
\]
and $\mathbf{p}\de\left\{\left(p_{l_1},\ldots,p_{l_k}\right)\colon p_{l_1}\leq X,\ldots,p_{l_k}\leq X\right\}$, $a_{\mathbf{p}}\de a_{p_{l_1}}\cdots a_{p_{l_k}}$, and $|\mathbf{p}|$ is the number of permutations of $p_{l_1},\ldots,p_{l_k}$.
\end{lemma}

\begin{proof}
The proof is the same as in \cite[pp.~374--375]{Tsang}, except that we used Lemma~\ref{lem:MVP}.
\end{proof}

Notice that
\begin{equation}
\label{eq:gbC1}
C_1(k) = \frac{3\pi m_0\sqrt{(2k)!}}{4^k}\sum_{m=0}^{2k}\sqrt{\binom{2k}{m}} \leq \frac{3\pi m_0\sqrt{2k+1}\sqrt{(2k)!}}{2^k} \leq 6\pi m_0k^k,
\end{equation}
where we used Stirling approximation $n!\leq 2\sqrt{2n}(n/e)^n$ for positive integers $n$.

Let $X\geq 2$. Lemma \ref{lem:tsang2} implies
\begin{equation}
\label{eq:main}
\left|\int_{T}^{2T} \Im\left\{\sum_{p\leq X} \frac{a_p}{p^{\ie t}}\right\}^{2}\dif{t} - \frac{T}{2}\sum_{p\leq X}\left|a_{p}\right|^2\right| \leq C_1(1)\sum_{p\leq X}p\left|a_p\right|^2,
\end{equation}
where $C_1(1)=3\pi m_0\left(1+\sqrt{2}\right)/2$. Since
\[
\sum_{p_l,p_m\leq X} \left|a_{p_l}a_{p_m}\right|^2 \left|\left\{p_l,p_m\right\}\right|! = 2\left(\sum_{p\leq X}\left|a_{p}\right|^2\right)^2 - \sum_{p\leq X}\left|a_p\right|^4,
\]
it also implies that
\begin{multline}
\label{eq:main2}
\left|\int_{T}^{2T} \Im\left\{\sum_{p\leq X} \frac{a_p}{p^{\ie t}}\right\}^{4}\dif{t} - \frac{3T}{4}\left(\sum_{p\leq X}\left|a_{p}\right|^2\right)^2\right| \leq \frac{3T}{4}\sum_{p\leq X}\left|a_p\right|^4 \\
+ C_1(2)\left(\sum_{p\leq X}p\left|a_p\right|^2\right)^2,
\end{multline}
where $C_1(2)=9\pi m_0\left(1+\sqrt{6}\right)/4$.

\begin{lemma}
\label{lem:primes}
Let $X\geq2$. If $\log{2}/\log{X}\leq h\leq 1$, then
\[
\left|\log{\left(h\log{X}\right)} - \sum_{p\leq X} \frac{1-\cos{\left(h\log{p}\right)}}{p}\right| \leq 13.88 + \frac{3}{\log^2{X}}.
\]
If $0\leq h\leq\log{2}/\log{X}$, then
\[
\sum_{p\leq X} \frac{1-\cos{\left(h\log{p}\right)}}{p} \leq 2.02 + \frac{3}{\log^2{X}}.
\]
\end{lemma}

\begin{proof}
By partial summation we have
\[
\sum_{p\leq X} f(p) = \int_{2}^{X} \frac{f(y)}{\log{y}}\dif{y} + \frac{2f(2)}{\log{2}} + \frac{f(X)\left(\vartheta(X)-X\right)}{\log{X}}-\int_{2}^{X} \left(\vartheta(y)-y\right)\frac{\dif{}}{\dif{y}}\frac{f(y)}{\log(y)}\dif{y}
\]
for a differentiable function $f(y)$, see \cite[Equation 4.14]{RosserSchoenfeld}. We will apply this equation on $f(y)\de y^{-1}\left(1-\cos{(h\log{y})}\right)$. Then
\[
\int_{2}^{X} \frac{f(y)}{\log{y}}\dif{y} = \Cin(h\log{X}) - \Cin(h\log{2}),
\]
where
\[
\Cin(z) = \int_0^z \frac{1-\cos{t}}{t}\dif{t}
\]
is one of the cosine integrals. Since $\Cin(z)$ is a positive function for $z\in\R^+$, it follows that $0<\Cin{(\log{2})}-\Cin(h\log{2})\leq \Cin{(\log{2})}$. Integration by parts gives
\[
\left|\Cin{(h\log{X})} - \Cin{\log{2}} - \log{\left(h\log{X}\right)}\right| \leq \left|\log{\log{2}}\right| +2+\frac{1}{\log{2}}.
\]
This implies
\[
\left|\log{\left(h\log{X}\right)} - \int_{2}^{X} \frac{f(y)}{\log{y}}\dif{y}\right| \leq 3.927.
\]
If $h\log{X}<\log{2}$, then
\[
\left|\int_{2}^{X} \frac{f(y)}{\log{y}}\dif{y}\right| \leq \Cin(\log{2})<0.118.
\]

By \cite[Equations 3.15 and 3.16]{RosserSchoenfeld} and numerical verification up to $X\leq41$, we have $\left|\vartheta(X)-X\right| \leq 3X/\left(2\log{X}\right)$ for all $X\geq 2$. This implies
\begin{gather*}
\left|\int_2^X \left(\vartheta(y)-y\right)\frac{\dif{}}{\dif{y}}\frac{f(y)}{\log(y)}\dif{y}\right| \leq \frac{3}{2}\int_{2}^{\infty} \frac{2+3\log{y}}{y\log^3{y}}\dif{y} < 9.62, \\
\left|\frac{f(X)\left(\vartheta(X)-X\right)}{\log{X}}\right| \leq \frac{3}{\log^2{X}}.
\end{gather*}
We also have
\[
\left|\frac{2f(2)}{\log{2}}\right| \leq \frac{1-\cos{\left(\log{2}\right)}}{\log{2}} < 0.333.
\]
In the case $h\log{X}<\log{2}$ we further have
\[
\left|\int_2^X \left(\vartheta(y)-y\right)\frac{\dif{}}{\dif{y}}\frac{f(y)}{\log(y)}\dif{y}\right| \leq \frac{3}{2}\int_{2}^{\infty} \frac{0.6737+0.2308\log{y}}{y\log^3{y}}\dif{y} < 1.56.
\]
All bounds give the final estimates from the lemma.
\end{proof}

The next proposition may be regarded as an uniform version of Lemma \ref{lem:primes}.

\begin{proposition}
\label{prop:sum}
Let $X\geq 2$, $0\leq h\leq 1$ and $0<a\leq7/10$. Then
\begin{gather*}
\left|\log{\left(a+h\log{X}\right)} - \sum_{p\leq X} \frac{1-\cos{\left(h\log{p}\right)}}{p}\right| \leq A, \\
\left|\log^2{\left(a+h\log{X}\right)} - \left(\sum_{p\leq X} \frac{1-\cos{\left(h\log{p}\right)}}{p}\right)^2\right| \leq 2A\left|\log{\left(a+h\log{X}\right)}\right|+A^2,
\end{gather*}
where
\[
A(a,X)\de \left|\log{a}\right| + 13.88 + \log{\left(1+\frac{a}{\log{2}}\right)} + \frac{3}{\log^2{X}}.
\]
\end{proposition}

\begin{proof}
Observe that
\[
\left|\log{\left(a+h\log{x}\right)} - \log{\left(h\log{x}\right)}\right| \leq \log{\left(1+\frac{a}{\log{2}}\right)}
\]
if $h\log{x}\geq\log{2}$, and
\[
\left|\log{\left(a+h\log{x}\right)}\right| \leq \max\left\{\left|\log{a}\right|,\left|\log{\left(a+\log{2}\right)}\right|\right\}=\left|\log{a}\right|
\]
if $h\log{x}\leq\log{2}$. Now the first bound follows from Lemma \ref{lem:primes}. The second bound easily follows from the first bound.
\end{proof}

\subsection{Proof of Theorem \ref{thm:difference}}

Define
\begin{gather*}
    Q_{k}(t)\de S(t) + \frac{1}{\pi}\sum_{p\leq T^{\frac{\varepsilon}{k}}}\frac{\sin{(t\log{p})}}{\sqrt{p}} = S(t) - \frac{1}{\pi}\sum_{p\leq T^{\frac{\varepsilon}{k}}}\Im\left\{\frac{p^{-\frac{1}{2}}}{p^{\ie t}}\right\}, \\
    P_{k}(t) \de \frac{1}{\pi}\sum_{p\leq T^{\frac{\varepsilon}{k}}}\Im\left\{\frac{p^{-\frac{1}{2}}\left(p^{-\ie h}-1\right)}{p^{\ie t}}\right\}=Q_{k}(t)-S(t)+S(t+h)-Q_{k}(t+h).
\end{gather*}
Take $a_p\de\pi^{-1}p^{-1/2}\left(p^{-\ie h}-1\right)$. Then $\left|a_p\right|^2=2\pi^{-2}p^{-1}\left(1-\cos{\left(h\log{p}\right)}\right)$, $p\left|a_p\right|^2\leq \left(2/\pi\right)^2$ and $\left|a_p\right|^4\leq \left(2/\pi\right)^4p^{-2}$.

Let $X=T^{\varepsilon}$ and $a=\alpha\varepsilon$. Inequality \eqref{eq:main} and Proposition \ref{prop:sum} give
\[
\left|\int_{T}^{2T}P_1^2(t)\dif{t} - \frac{1}{\pi^2}T\log{\left(\alpha+h\log{T}\right)}\right|
\leq D_1T,
\]
where
\begin{equation}
\label{eq:d1}
D_1 \de \frac{1}{\pi^2}\left(\left|\log{\varepsilon}\right|+A\left(\alpha\varepsilon,T_{1}^{\varepsilon}\right)+\frac{4C_1(1)}{T_{1}^{1-\varepsilon}}\right).
\end{equation}
Let $X=T^{\varepsilon/2}$ and $a=\alpha\varepsilon/2$. Similarly, by \eqref{eq:main2} and Proposition \ref{prop:sum} we also have
\[
\left|\int_{T}^{2T}P_2^4(t)\dif{t} - \frac{3}{\pi^4}T\log^2{\left(\alpha+h\log{T}\right)}\right| \leq D_2T\log{\left(\alpha+h\log{T}\right)},
\]
where
\begin{multline}
\label{eq:d2}
D_2 \de \frac{6}{\pi^4}\left(A\left(\frac{\alpha\varepsilon}{2},T_1^{\frac{\varepsilon}{2}}\right)+\left|\log{\frac{\varepsilon}{2}}\right|\right) \\
+\frac{3}{\pi^4\log\left(\alpha+h\log{T}\right)}\left(A\left(\frac{\alpha\varepsilon}{2},T_1^{\frac{\varepsilon}{2}}\right)^2+\frac{3\pi^4}{160}+\frac{16C_1(2)}{3T_1^{1-\varepsilon}}\right).
\end{multline}

Define $\Delta{S}(t)\de S(t+h)-S(t)$ and $\Delta{Q_k}(t)\de Q_k(t+h)-Q_k(t)$. Because $\Delta{S}(t)=P_k(t)+\Delta{Q_k}(t)$, we have by the Cauchy--Schwarz inequality
\begin{flalign*}
\left|\int_{T}^{2T}\left|\Delta{S}(t)\right|^2-P_1^2(t)\dif{t}\right| &\leq \int_{T}^{2T} \left|\Delta{Q_1}(t)\right|^2\dif{t} \\
&+2\left(\int_{T}^{2T}P_1^2(t)\dif{t}\right)^{\frac{1}{2}}\left(\int_{T}^{2T} \left|\Delta{Q_1}(t)\right|^2\dif{t}\right)^{\frac{1}{2}}.
\end{flalign*}
Similarly, by H\"{o}lder's inequality we also obtain
\begin{flalign*}
\left|\int_{T}^{2T}\left|\Delta{S}(t)\right|^4-P_2^4(t)\dif{t}\right| &\leq \int_{T}^{2T} \left|\Delta{Q_2}(t)\right|^4\dif{t} \\
&+ 4\left(\int_{T}^{2T}P_2^4(t)\dif{t}\right)^{\frac{3}{4}}\left(\int_{T}^{2T} \left|\Delta{Q_2}(t)\right|^4\dif{t}\right)^{\frac{1}{4}} \\
&+ 6\left(\int_{T}^{2T}P_2^4(t)\dif{t}\right)^{\frac{1}{2}}\left(\int_{T}^{2T} \left|\Delta{Q_2}(t)\right|^4\dif{t}\right)^{\frac{1}{2}} \\
&+ 4\left(\int_{T}^{2T}P_2^4(t)\dif{t}\right)^{\frac{1}{4}}\left(\int_{T}^{2T} \left|\Delta{Q_2}(t)\right|^4\dif{t}\right)^{\frac{3}{4}}.
\end{flalign*}
Inequality \eqref{eq:corHolder} guarantees, together with Theorem \ref{thm:mainForS} and Corollary \ref{cor:mainForS}, that we have
\[
\int_{T}^{2T} \left|\Delta{Q_1}(t)\right|^2\dif{t} \leq E_2T, \quad
\int_{T}^{2T} \left|\Delta{Q_2}(t)\right|^4\dif{t} \leq F_2T.
\]
Because
\[
\int_{T}^{2T}P_1^2(t)\dif{t} \leq E_1T\log{\left(\alpha+h\log{T}\right)}, \quad \int_{T}^{2T}P_2^4(t)\dif{t} \leq F_1T\log^2{\left(\alpha+h\log{T}\right)},
\]
the proof now easily follows.

\section{The density of zeroes of a given multiplicity and the proof of Theorem \ref{thm:multMain}}\label{densityproof}

Theorem \ref{thm:mainForS}, together with Lemma \ref{lem:tsang2}, enables us to prove the following upper bound for the density of the nontrivial zeroes $\rho$ of given multiplicity $j\geq 1$ with $T<\Im\left\{\rho\right\}\leq 2T$, where $T$ is sufficiently large.

\begin{theorem}
\label{thm:mult}
There exists $T_0>0$ such that
\begin{equation*}
\label{table}
\frac{N_j(2T)-N_j(T)}{N(2T)-N(T)} \leq 1.01395\cdot e^{-6.459\cdot10^{-7}j}
\end{equation*}
for every $j\geq1$ and $T\geq T_0$.
\end{theorem}

Fujii \cite{FujiiDistr} applied general bounds for $\mathcal{J}_{2k}(T,h)$ to obtain that the number of zeroes $\rho$ with $0<\Im\left\{\rho\right\}\leq T$ and multiplicity at least $j>j_0$ for some $j_0\geq1$ is at most $e^{-A\sqrt{j}}N(T)$ for some positive constant $A$. Later, he improved this result to $e^{-Aj}N(T)$, see \cite[Theorem 3]{Fujii}. Karatsuba and Korol\"{e}v \cite[Theorem 7]{KaratKor} proved that
\[
\frac{N_j(T+H)-N_j(T)}{N(T+H)-N(T)} \leq e^{7.2}\cdot \exp{\left(-\frac{2\varepsilon\sqrt{\varepsilon}}{e\sqrt{10e^{37}\pi^{-2}}}j\right)},
\]
where $0<\varepsilon<10^{-3}$, $H=T^{\frac{27}{82}+\varepsilon}$, $j\geq1$, and $T$ is sufficiently large. This estimate gives a bound as in Theorems \ref{thm:multMain} and \ref{thm:mult}, but with worst constants. Our improvement upon the constant in the exponent comes from the second part of Corollary \ref{cor:mainValues1}, while a suitable choice for $h$ in $\mathcal{J}_{2k}(T,h)$ allows to reduce the factor $e^{7.2}$ significantly. Otherwise our proof follows the same ideas outlined in \cite{KaratKor}.

\begin{lemma}
\label{lem:gbJ}
Let $\eta\geq\left(5\omega_0\right)^{-1}$ with $\omega_0$ as in Corollary \ref{cor:mainValues1}, and define
\begin{equation}
\label{eq:h}
h\left(T,\eta\right) \de \frac{2\pi}{\eta}\left(\log{\frac{T}{2\pi}}\right)^{-1}.
\end{equation}
Then there exists $T_0>0$ such that $\mathcal{J}_{2k}\left(T,2h\left(T,\eta\right)\right)\leq \left(3\omega_0k\right)^{2k}T$ for
$1\leq k \leq \left(176\cdot10^4\right)^{-1}\log{T}$ and $T\geq T_0$.
\end{lemma}

\begin{proof}
Take $T$ sufficiently large such that $h=h\left(T,\eta\right)\in[0,1/2]$. Note that this is independent of $\eta$ since this quantity is bounded below by a fixed constant. Let $\varepsilon=3/88$. Using the same notation as in the proof of Theorem \ref{thm:difference}, from Theorem \ref{thm:mainForS}, Corollary \ref{cor:mainForS}, and Corollary \ref{cor:mainValues1} we obtain
\begin{flalign*}
\mathcal{J}_{2k}\left(T,2h\right) &\leq \int_{T}^{2T} \left(\left|P_k(t)\right|+\left|Q_k(t+2h)\right|+\left|Q_k(t)\right|\right)^{2k} \dif{t} \\
&\leq \frac{2}{3}\left(3\omega_0k\right)^{2k}T + 3^{2k-1}\int_{T}^{2T}\Im\left\{\sum_{p\leq T^{\frac{\varepsilon}{k}}}\frac{a_p}{p^{\ie t}}\right\}^{2k}\dif{t}
\end{flalign*}
for $1\leq k \leq \left(176\cdot10^4\right)^{-1}\log{T}$. By Lemma \ref{lem:tsang2} and the inequality \eqref{eq:gbC1} we also have
\[
\int_{T}^{2T}\Im\left\{\sum_{p\leq T^{\frac{\varepsilon}{k}}}\frac{a_p}{p^{\ie t}}\right\}^{2k}\dif{t} \leq k^k\left(\sum_{p\leq T^{\frac{\varepsilon}{k}}}\left|a_p\right|^2\right)^k T + 6\pi m_0k^k\left(\sum_{p\leq T^{\frac{\varepsilon}{k}}}p\left|a_p\right|^2\right)^k,
\]
where we used the simple bound
\[
\frac{k!}{4^k}\binom{2k}{k} \leq k! \leq k^k.
\]
Because
\[
\left|a_p\right|^2 = \left(\frac{2}{\pi}\right)^2 \frac{\sin^2\left(\frac{h\log{p}}{2}\right)}{p} \leq \left(\frac{h}{\pi}\right)^2 \frac{\log^2{p}}{p},
\]
we obtain
\begin{gather*}
\sum_{p\leq T^{\frac{\varepsilon}{k}}}p\left|a_p\right|^2 \leq \left(\frac{h\varepsilon}{\pi k}\log{T}\right)^{2} T^{\frac{\varepsilon}{k}}, \\
\sum_{p\leq T^{\frac{\varepsilon}{k}}}\left|a_p\right|^2 \leq \left(\frac{h}{\pi}\right)^2\frac{\varepsilon}{k}\log{T}\sum_{p\leq T^{\frac{\varepsilon}{k}}}\frac{\log{p}}{p} \leq \left(\frac{h\varepsilon}{\pi k}\log{T}\right)^{2},
\end{gather*}
by (\ref{desk}).
This implies
\[
3^{2k-1}\int_{T}^{2T}\Im\left\{\sum_{p\leq T^{\frac{\varepsilon}{k}}}\frac{a_p}{p^{\ie t}}\right\}^{2k}\dif{t} \leq \left(\frac{3h\varepsilon}{\pi\sqrt{k}}\log{T}\right)^{2k}T \leq \frac{1}{3}\left(3\omega_0k\right)^{2k}T
\]
for sufficiently large $T$. This concludes the proof of the lemma.
\end{proof}

In what follows, we are going to use $\left|S(T)\right|\leq \frac{1}{5}\log{T}$ and $N(2T)-N(T)\geq \frac{T}{2\pi}\log{\frac{T}{2\pi}}$, both valid for sufficiently large $T$, see, e.g., Theorem 1 and Corollary 1 in \cite{Tru14}. Also, denote by $m\left(\cdot\right)$ the measure of a subset in $\R$.

\begin{lemma}
\label{lem:D}
For $\lambda\geq0$ define the set
\[
D(\lambda)\de\left\{t\in[T,2T]\colon \left|S(t+2h\left(T,\eta\right))-S(t)\right|\geq\lambda\right\},
\]
where $h\left(T,\eta\right)$ is defined by \eqref{eq:h}. Then there exists $T_0>0$ such that
\[
m\left(D(\lambda)\right) \leq T\exp{\left(4-\frac{2}{3e\omega_0}\lambda\right)}
\]
for $T\geq T_0$, with $\omega_0$ as in Corollary \ref{cor:mainValues1}.
\end{lemma}

\begin{proof}
Write $h=h\left(T,\eta\right)\in[0,1/2]$. We are going to consider three different cases: $\lambda\geq \frac{1}{2}\log{T}$, $0\leq\lambda\leq6e\omega_0$, and $6e\omega_0<\lambda<\frac{1}{2}\log{T}$. In the first case we have $m\left(D(\lambda)\right)=0$ since
\[
\left|S(t+2h)-S(t)\right| \leq \frac{2}{5}\log{T} + \frac{2}{5}\log{2} + \frac{1}{5T} < \frac{1}{2}\log{T}
\]
for sufficiently large $T$. In the second case we obtain the trivial bound $m\left(D(\lambda)\right)\leq T$. Moving to the third case, define $k\de\lfloor\lambda/\left(3e\omega_0\right)\rfloor$. Such $k$ satisfies the conditions of Lemma \ref{lem:gbJ}, which implies that we have
\begin{flalign*}
m\left(D(\lambda)\right) &\leq \lambda^{-2k}\int_{D(\lambda)}\left(S(t+2h)-S(t)\right)^{2k}\dif{t} \\
&\leq \lambda^{-2k}\mathcal{J}_{2k}\left(T,2h\right)\leq \left(\frac{3\omega_0}{\lambda}k\right)^{2k}T = \left(\frac{3e\omega_0k}{\lambda}\right)^{2k}T\exp{\left(-2k\right)}
\end{flalign*}
for sufficiently large $T$. Because $3e\omega_0k/\lambda\leq 1$ and $-2k\leq 2-2\lambda/\left(3e\omega_0\right)$, the stated inequality from the lemma is also true in this case. The proof is thus complete.
\end{proof}

\begin{lemma}
\label{lem:lbS}
Let $\gamma$ be an ordinate of a nontrivial zero of $\zeta(s)$, such that $$T\leq \gamma-h\left(T,\eta\right)<\gamma\leq2T,$$ where $h\left(T,\eta\right)$ is defined by \eqref{eq:h}, and let the interval $\left(\gamma-h\left(T,\eta\right),\gamma\right]$ contain exactly $\nu$ ordinates of zeroes of multiplicities $j\geq1$. Then the inequality
\[
S\left(t+2h\left(T,\eta\right)\right) - S(t) \geq \nu j - \frac{2}{\eta} - \varepsilon
\]
holds for any $t\in\left(\gamma-2h\left(T,\eta\right),\gamma-h\left(T,\eta\right)\right]$, where $T\geq T_0(\varepsilon)>0$ and $\varepsilon>0$.
\end{lemma}

\begin{proof}
The proof is the same as in \cite[Lemma 16]{KaratKor}, just with $H=T$ and our definition of $h$.
\end{proof}

\subsection{Proof of Theorem \ref{thm:mult}}
Define $\kappa\de2/\left(3e\omega_0\right)$. Let $\delta>1$. We are going to consider three different cases: $j\geq\frac{1}{2}\log{T}$, $0\leq j\leq\kappa^{-1}\log{\delta}$, and $\kappa^{-1}\log{\delta}<j<\frac{1}{2}\log{T}$. In the first case we claim that $N_j(T)=N_j(2T)=0$. For, suppose not, and let $\rho=\beta+\ie\gamma$, $\gamma\leq 2T$, be a zero of multiplicity $j$. Then
\[
j\leq S\left(\gamma+\varepsilon\right)-S\left(\gamma-\varepsilon\right)<\frac{1}{2}\log{T},
\]
which is in contradiction with the first condition. In the second case we obtain
\[
N_j(2T)-N_j(T) \leq \delta\left(N(2T)-N(T)\right)e^{-\kappa j}.
\]
Moving to the third case, let us assume that $N_j(2T)-N_j(T)>0$ since otherwise the stated inequality is trivial. Also $h=h\left(T,\eta\right)\in[0,1/2]$. Let $\gamma_1$ be the largest value among the ordinates of the nontrivial zeroes $\rho$ with $T<\Im\left\{\rho\right\}\leq2T$ and multiplicity at least $j$. Let $\mathcal{E}_1\de\left(\gamma_1-h,\gamma_1\right]$, and denote by $\gamma_2$ the largest value among the ordinates of the nontrivial zeroes $\rho$ with $T<\Im\left\{\rho\right\}\leq\gamma_1-h$ and multiplicity at least $j$, if such value exists, and put $\mathcal{E}_2\de\left(\gamma_2-h,\gamma_2\right]$. We continue this process with quantities $\gamma_3,\gamma_4,\ldots$ and the intervals $\mathcal{E}_3,\mathcal{E}_4,\ldots$ until there are no ordinates of zeroes with the above property on the interval $(T,\gamma_k-h]$ or we obtain an ordinate $\gamma_k\in\left[T,T+h\right)$. Such intervals are pairwise disjoint, have the same length $h$, and their union covers all ordinates of $\rho$ with multiplicity at least $j$. We partition these intervals into classes $\mathscr{E}_1,\mathscr{E}_2,\ldots,\mathscr{E}_N$ by taking into $\mathscr{E}_n$ those intervals containing exactly $n$ ordinates of $\rho$ with multiplicity at least $j$. Then we must have
\[
N_j(2T)-N_j(T) \leq \sum_{n=1}^{N} n\left|\mathscr{E}_n\right|.
\]
Suppose that $\mathcal{E}_k$ belongs to $\mathscr{E}_n$. Taking $\mathcal{E}'_k\de\mathcal{E}_k-h=\left(\gamma_k-2h,\gamma_k-h\right]$, it follows by Lemma \ref{lem:lbS} that $S\left(t+2h\right) - S(t) \geq nj - 2/\eta - \varepsilon$ for any $t\in\mathcal{E}'_k$. In the notation of Lemma \ref{lem:D} this means that $\mathcal{E}'_k\subseteq D\left(nj - 2/\eta - \varepsilon\right)$. Because this is true for every shifted interval which belongs to $\mathscr{E}_n$, and these intervals are pairwise disjoint having the same length $h$, we obtain by Lemma \ref{lem:D} that
\[
h\left|\mathscr{E}_n\right|\leq m\left(D\left(nj - \frac{2}{\eta} - \varepsilon\right)\right) \leq T\exp{\left(4+\frac{2\kappa}{\eta}+\varepsilon\kappa-n\kappa j\right)}.
\]
Because $N(2T)-N(T)\geq T/\left(h\eta\right)$ for sufficiently large $T$, the previous two inequalities imply
\begin{flalign*}
\frac{N_j(2T)-N_j(T)}{N(2T)-N(T)} &\leq \eta\exp{\left(4+\frac{2\kappa}{\eta}+\varepsilon\kappa\right)}\sum_{n=1}^{\infty} ne^{-n\kappa j} = \frac{\eta\exp{\left(4+\frac{2\kappa}{\eta}+\varepsilon\kappa\right)}}{\left(1-e^{-\kappa j}\right)^2}e^{-\kappa j} \\
&< \frac{\eta\exp{\left(4+\frac{2\kappa}{\eta}+\varepsilon\kappa\right)}}{\left(1-1/\delta\right)^2}e^{-\kappa j},
\end{flalign*}
where we used also the condition $\kappa^{-1}\log{\delta}<j$. The choice $\eta=2\kappa$, which minimizes the numerator in the above fraction, satisfies the condition in Lemma \ref{lem:gbJ}. Because the appropriate solution of the equation $2e^{5}\kappa=\delta\left(1-1/\delta\right)^2$ is $\delta\approx1.013943$, and $\kappa>6.459\cdot10^{-7}$, the inequality of Theorem \ref{thm:mult} now easily follows.

\subsection{Proof of Theorem \ref{thm:multMain}}
After dyadic partition and using the bound $j<\log{T}$ proved above, Theorem \ref{thm:mult} implies that
\[
\frac{N_j(T)}{N(T)} \leq \left(1.01395+\frac{N\left(2T_0\right)T^{6.459\cdot10^{-7}}}{N(T)}\right)e^{-6.459\cdot10^{-7} j}
\]
for $T\geq 2T_0$, where $T_0$ is fixed constant from Theorem \ref{thm:mult}. But $N(T)\sim\frac{T}{2\pi}\log{T}$, which means that the fraction in the above parentheses can be made arbitrarily small after taking $T$ sufficiently large. This finishes the proof of Theorem \ref{thm:multMain}.

\section*{Acknowledgements}
We are grateful to Dan Goldston for his helpful comments.


\providecommand{\bysame}{\leavevmode\hbox to3em{\hrulefill}\thinspace}
\providecommand{\MR}{\relax\ifhmode\unskip\space\fi MR }
\providecommand{\MRhref}[2]{%
  \href{http://www.ams.org/mathscinet-getitem?mr=#1}{#2}
}
\providecommand{\href}[2]{#2}

\end{document}